\documentclass[11pt,a4paper,english,reqno]{amsart}
\usepackage{amsmath,amssymb,amsfonts,epsfig,mathrsfs}
\usepackage[T1]{fontenc}

\usepackage{color}
\usepackage{array}
\usepackage{amsthm}
\usepackage{amstext}
\usepackage{graphicx}
\usepackage{setspace}
\usepackage[margin=2.5cm]{geometry}
\usepackage{bbm}
\usepackage{color}
\usepackage{enumitem}
\usepackage{undertilde}
\usepackage{amscd,psfrag}
\usepackage{yhmath}
\usepackage[mathscr]{eucal}

\usepackage{slashed}

\makeatletter
\pdfpageheight\paperheight
\pdfpagewidth\paperwidth

\usepackage{epstopdf}
\usepackage{chngcntr}
\counterwithin{figure}{section}
\usepackage{mathrsfs}

\usepackage{indentfirst}	

\usepackage[normalem]{ulem}
\theoremstyle{plain}

\newtheorem{definition}{Definition}[section]
\newtheorem{theorem}[definition]{Theorem}
\newtheorem*{theorem*}{Theorem}

\newtheorem*{remark*}{Remark}
\newtheorem*{sideremark*}{Side Remark}

\newtheorem*{claim*}{Claim}
\newtheorem*{q*}{Question}
\newtheorem{lemma}[definition]{Lemma}

\newtheorem*{corollary*}{Corollary}

\newtheorem{proposition}[definition]{Proposition}

\newcommand{\rea}{\mathbb{R}}

\newcommand{\ir}{\int_{\mathbb{R}}}

\newcommand{\supt}{\sup_{[0,T]\times \mathbb{R}}\theta}

\newcommand{\dd}{{\rm d}}
\newcommand{\dx}{\,\dd x}
\newcommand{\dt}{\, \dd t}

\numberwithin{equation}{section}
\numberwithin{figure}{section}

\title{On One-Dimensional Compressible Navier-Stokes Equations For A Reacting Mixture In Unbounded Domains}
\author{Siran Li}
\address{Siran Li: Mathematical Institute, University of Oxford, Oxford, OX2 6GG, UK}
\email{\texttt{siran.li@maths.ox.ac.uk}}

\date{\today}

\begin{document}

\maketitle

\begin{abstract}

In this paper we consider the one-dimensional Navier-Stokes system for a heat-conducting, compressible reacting mixture  which describes the dynamic combustion of fluids of mixed kinds on unbounded domains. This  model has been discussed on bounded domains by Chen \cite{C92} and Chen-Hoff-Trivisa \cite{CHT03} among others, in which the reaction rate function is a discontinuous function obeying the Arrhenius Law. We prove the global existence of weak solutions to this model on one-dimensional unbounded domains with large initial data in $H^1$. Moreover, the large-time behaviour of the weak solution is identified and proved. In particular, the uniform-in-time bounds for the temperature and specific volume have been established via energy estimates. For this purpose we utilise techniques developed by Kazhikhov and coauthors ({\it cf.} \cite{K82,SK81}), as well as a crucial estimate in the recent work by Li-Liang \cite{LL14}. Several new estimates are also established, in order to treat the unbounded domain and the reacting terms.
\end{abstract}

\section{Introduction and Main Results}

The equations of motion for the  compressible fluids describing chemical reactions and radiative processes have been a heated research topic in fluid dynamics: {\it Cf.} \cite{LL14, C92, CHT03, CK, DT, DZ, wang} and the references cited therein. In the current work we are concerned with the global existence and large-time behaviour of global solutions to the compressible Navier-Stokes equations for a reacting mixture on one-dimensional unbounded domains. Our system describes the physical process of dynamic combustion, for which the reacting rate function is discontinuous and obeys the {\em Arrhenius Law} of molecular thermodynamics. 

Following Chen (\cite{C92}) in which the explicit transform from Euler to Lagrangian coordinates has been computed, in this paper our analysis for the compressible Navier-Stokes equations will be carried out in the Lagrangian coordinates, {\it i.e.},

\begin{equation}\label{eq_mass equation}
u_t - v_x =0,
\end{equation}
\begin{equation}\label{eq_momentum equation}
v_t+\big(\frac{a\theta}{u}\big)_x = \big(\frac{\mu v_x}{u}\big)_x,
\end{equation}
\begin{equation}\label{eq_temperature equation}
\big(\theta+\frac{v^2}{2}\big)_t + \big(\frac{av\theta}{u}\big)_x = \big(\frac{\mu vv_x + \kappa \theta_x}{u}\big)_x + qK\phi(\theta)Z,
\end{equation}
\begin{equation}\label{eq_reaction equation}
Z_t + K \phi(\theta)Z = \big(\frac{d}{u^2}Z_x\big)_x.
\end{equation}

In the above system we are solving for the four  dynamic variables $(u, v, \theta, Z)$, which represent the {\em specific volume, velocity, temperature, and mass fraction of the reactant}, respectively. The positive constants $\mu, \kappa, q, d, a$ and $K$ are the coefficients of bulk viscosity, heat conduction, species diffusion, difference in the internal energy of the reactant and the product, the product of Boltzmann's gas constant and the molecular weight, and the reaction rate, respectively. 

One distinctive feature of the above system consisting of Eqs. \eqref{eq_mass equation}--\eqref{eq_reaction equation} is the presence of $\phi(\theta)$, known as the {\em reaction rate function}.   Here $\phi:\rea\rightarrow [0,\infty)$ is a function of the temperature $\theta$ determined by the Arrhenius Law:
\begin{equation}
\phi(\theta)=\theta^\alpha e^{-\frac{A}{\theta}} \mathbbm{1}_{\{\theta>\theta_{\text{ignite}}\}},
\end{equation}
where $\alpha, A>0$ are thermodynamic constants, and $\theta_{\text{ignite}}>0$ is the threshold temperature which triggers the reaction. In particular, this function is discontinuous at $\theta_{\text{ignite}}$. To deal with the reaction rate function $\phi$, we first mollify it and derive uniform bounds for the resulting $C^1$ functions, and then pass to the limits to recover the discontinuous $\phi(\theta)$. Here we need the uniform boundedness of $\phi$, which is justified {\em a posteriori} via the uniform bounds for the other dynamical variables, {\it i.e.}, $(u,v,Z)$.

In this work we consider the Cauchy problem on the whole real line $\Omega = \rea$. More precisely, the {\em initial data} is prescribed as follows:
\begin{equation}\label{eq_initial data}
(u,v,\theta,Z)|_{t=0} = (u_0,v_0,\theta_0,Z_0),
\end{equation}
and the following {\em far-field condition} is imposed: 
\begin{equation}\label{eq_far-field condition}
\lim_{|x|\rightarrow \infty} (u,v,\theta,Z) (x,t) = (1,0,1,0) \qquad \text{ for all } t \geq 0.
\end{equation}
Physically, it means that at the endpoints of the reacting system the density is constant ({\it i.e.}, no formation of vacuum or density-concentration), and so is the temperature. Also, the endpoints are kept fixed for all the time, with no chemical reaction triggered there.

Moreover, the initial data are assumed to satisfy the following conditions: 
\begin{equation}\label{eq_conditions on id}
\begin{cases}
0<m_0\leq u_0(x),  \theta_0(x) \leq M_0 < \infty, \quad 0 \leq Z_0(x) \leq 1,\\
|v_0(x)| \leq M_0,\\
(u_0-1, v_0, \theta_0 - 1, Z_0) \in [H^1(\rea)]^4,
\end{cases}
\end{equation}
where $m_0, M_0$ are universal constants. The regularity condition in the last line is  referred to as the {\em large data} condition.

Now, let us introduce the notion of weak solutions to the compressible Navier-Stokes system of the reacting mixture, which is our main object of study in this work:

\begin{definition}\label{def_weak and classsical solutions}

The quadruplet $(u,v,\theta, Z):[0,T]\times\rea \rightarrow \rea^4$ is a {\bf weak solution} to the system \eqref{eq_mass equation}--\eqref{eq_conditions on id} if it satisfies the equations in the sense of distributions on $[0,T]\times\rea$, and satisfies the following regularity conditions:
\begin{equation*}
\begin{cases}
u-1 \in L^\infty(0,T; H^1(\rea)),\\
u_t \in L^2(0,T;L^2(\rea)),\\
v, \theta-1, Z \in L^\infty(0,T;H^1(\rea))\cap L^2(0,T; H^2(\rea));\\
v_t, \theta_t, Z_t \in L^2(0,T; L^2(\rea)).
\end{cases}
\end{equation*}
\end{definition}

The main results of the paper are summarised as follows:

First, assuming the local (in time) existence of weak solutions, we prove the global existence of weak solutions to Eqs.  \eqref{eq_mass equation}-\eqref{eq_conditions on id}. Along the way, the {\em uniform bounds} (in space-time) for the temperature and the specific volume are established:

\begin{theorem}\label{Theorem for global existence of weak solution}
	There exists a weak solution $(u,v,\theta,Z)$ to Eqs. \eqref{eq_mass equation}-\eqref{eq_conditions on id} on $[0,T]\times \rea$ for all $T>0$. Moreover, there is a universal constant 
	\begin{equation*}
C_0=C_0\Big(a, \mu, \kappa, q, K, d, \phi(\cdot), \|(u_0-1, v_0, \theta_0 - 1 , Z_0)\|_{H^1(\rea)}, \inf_{\rea}u_0, \inf_{\rea}\theta_0\Big)
	\end{equation*}
such that
	\begin{equation}
0 < C_0^{-1} \leq \theta(t,x), u(t,x) \leq C_0 < \infty  \quad \text{ and }\quad  0\leq Z(t,x) \leq 1
	\end{equation}
for almost all $(t,x) \in [0,T] \times \rea$. In particular, $C_0$ is independent of $T$.
\end{theorem}

Meanwhile, the asymptotic states as $t\rightarrow \infty$, {\it i.e.} the {\em large-time behaviour}, of the reacting mixture, can be fully determined:

\begin{theorem}\label{thm_large time behaviour}
Let $(u,v,\theta, Z)$ be a global weak solution to Equations \eqref{eq_mass equation}-\eqref{eq_conditions on id}. Then it converges in $H^1$ to the equilibrium state in the far-field, {\it i.e.},
\begin{equation}\label{eqn_large time behaviour} \Big\|\Big(u(t,\cdot)-1,v(t,\cdot),\theta(t,\cdot)-1, Z(t,\cdot)\Big)\Big\|_{H^1(\rea)} \longrightarrow 0 \qquad \text{ as } t \rightarrow \infty.
\end{equation}
\end{theorem}

The remaining parts of the paper are organised as follows:

In \S 2 we collect several auxiliary conserved quantities and monotonicity formulae for the reacting mixture, which will be used throughout the paper. We also prove $0\leq Z\leq 1$. In \S 3 we establish the upper and lower bounds for the specific volume $u$. Next, in \S 4,  following the arguments in \cite{LL14} we derive uniform estimates involving $v, \theta$ and their first derivatives. 
Finally, in \S 5 we derive the upper and lower bounds for $\theta$ uniformly in space-time, together with the uniform bounds for higher derivatives of $(u,v,\theta,Z)$, and thus conclude the proof of Theorems \ref{Theorem for global existence of weak solution} and \ref{thm_large time behaviour}.

Before further development, we point out that the key estimate in this work, {\it i.e.}, Theorem \ref{thm_crucial estiamte in theta and u, Li Liang type}, essentially relies on the arguments in the recent paper \cite{LL14} by J. Li and Z. Liang, which in turn is motivated by the work of Huang-Li-Wang (\cite{HLW}) on a blowup criterion for compressible Euler equations. The new feature of our work lies in the physical process of dynamic combustions, {\it i.e.}, the analysis of  functions $\phi$ and $Z$, as well as the treatment for unbounded domains.

\section{Conserved Quantity and Entropy Formula}

In this section we record the conserved quantity and monotonicity formula of the compressible reacting mixture for future development. First of all, we have:

\begin{proposition}\label{proposition_energy and entropy} 
Let $(u,v,\theta,Z)$ be a weak solution on $[0,T]\times\rea$. Then there holds
\begin{equation}\label{eq_Z inequality}
\sup_{0\leq t \leq T}\ir Z(t,x){\rm d}x + \int_0^T\ir K\phi(\theta)Z \,{\rm d}x{\rm d}t\leq \ir Z_0(x)\,{\rm d}x =: E_0 <\infty.
\end{equation}
\end{proposition}

\begin{proof}
	Let us multiply $\beta Z^{\beta-1}$ with $\beta >1$ to the evolution equation of $Z$, {\em i.e.}, Eq. \eqref{eq_reaction equation}, to get
	\begin{equation}\label{identity for Z}
	\frac{d}{dt}\ir Z^\beta(t,x){\rm d}x + \ir pK\phi(\theta) Z^\beta {\rm d}x= -\ir\frac{\beta d(\beta-1)}{v^2}Z^{\beta-2}(Z_x)^2 \leq 0.
	\end{equation}
Thus, in view of $0 \leq Z \leq 1$ (which shall be established in Lemma \ref{lemma_bdd of Z}), we integrate from $0$ to $t$ and send $\beta \rightarrow 1^{+}$ to obtain Eq. \eqref{eq_Z inequality}, using the Dominated Convergence Theorem.
	
\end{proof}

Let us remark that, in Proposition \ref{proposition_energy and entropy} above we do not have the conservation of total mass or energy, as they may become unbounded. For instance, let us consider the reacting system of only one type of perfect gas, which it obeys the same $\gamma$-law (where $\gamma>1$ is a constant). In this case, the internal energy $e = pu/(\gamma-1)$ is proportional to the temperature $\theta  = pu/a$, thus the total energy of the reacting gas is 
\begin{equation*}
\ir \Big(\theta(t,x) + \frac{v(t,x)^2}{2}+ qZ(t,x)\Big) \,{\rm d}x.
\end{equation*}
However, in view of our far-field condition (Eq.  \eqref{eq_far-field condition}), $\theta \equiv 1$ is expected to be a steady state solution, which shall be verified later by the large-time behaviour ({\it cf.} Theorem \ref{thm_large time behaviour}). Such $\theta$ leads to infinite total energy. Similarly, $u_0 \equiv 1$ gives infinite total mass.

Next, we verify that $Z$ is indeed a ratio, {\it i.e.} a number between $0$ and $1$:

\begin{lemma}\label{lemma_bdd of Z}
Let $(u,v,\theta,Z)$ be a weak solution on $[0,T]\times \rea$. Then $0 \leq Z(t,x) \leq 1$ on $[0,T]\times \rea$.
\end{lemma}

\begin{proof}
	The proof for $Z \geq 0$ follows from the maximum principle. We set
	\begin{equation}
	Y(t,x):=e^{-\beta t} Z(t,x),
	\end{equation}
where $\beta >0$ is to be determined. Then, in view of Eq. \eqref{eq_reaction equation},  $Y$ satisfies the following evolution equation:
	\begin{equation}\label{interm eq_equation for Y}
	Y_t+[\beta+K\phi(\theta)] Y = (\frac{d}{v^2}Y_x)_x.
	\end{equation}
Here the infimum of $Y$ is attained on $\rea$, thanks to the far-field condition $\lim_{|x|\rightarrow\infty} Y(\cdot,x) = 0$. 

Now, suppose there were $(t_0,x_0)\in[0,T]\times \rea$ such that $Y(t_0,x_0)=\inf_{[0,T]\times\rea} Y<0$. Then it follows that
	\begin{equation*}
Y_x(t_0,x_0) = 0; \quad Y_t(t_0,x_0 ) \leq 0; \quad Y_{xx} (t_0,x_0) \geq 0,
	\end{equation*}
which contradicts Eq. \eqref{interm eq_equation for Y}. Thus we get $Y \geq 0$, which is equivalent to $Z \geq 0$. As a remark, here we need the requirement $Z_t \in L^2(0, T; L^2(\mathbb{R}))$ in Definition \ref{def_weak and classsical solutions} to ensure that $Y_t$ is well-defined.

	To prove the upper bound for $Z$, let us invoke again Eq. \eqref{identity for Z} (reproduced below):
\begin{equation*}
	\frac{d}{dt}\ir Z^\beta(t,x){\rm d}x + \ir pK\phi(\theta) Z^\beta {\rm d}x= -\ir\frac{\beta d(\beta-1)}{v^2}Z^{\beta-2}(Z_x)^2{\rm d}x.
\end{equation*}	
Since the right-hand side and the second term on the left-hand side are non-positive, the $L^\beta$-norm of $Z$ is decreasing in time for all $\beta \in [1,\infty)$. Thus, using the initial condition $0 \leq Z_0 \leq 1$ and sending $\beta \rightarrow \infty$, one immediately deduces $Z \leq 1$. Hence the assertion follows.

\end{proof}

Now we establish an important {\em monotonicity formula}, which is interpreted as the entropy/energy formula for the reacting mixture, referred to as the ``entropy inequality'' or ``entropy formula'' in the sequel. In  physics, the expressions $\big(u-1-\log(u)\big)$ and $\big(\theta-1-\log(\theta)\big)$ consist of the relative entropy, which obeys the {\it Clausius-Duhem} inequality of thermodynamics. We refer the readers to the appendix in \cite{CHT03} for a discussion on the relevant physical backgrounds. 

\begin{proposition}[Entropy Inequality]\label{proposition_eq_entropy ineq}
Let $(u,v,\theta,Z)$ be a weak solution on $[0,T]\times\rea$. Then the following inequality holds:
	\begin{align}\label{eq_entropy ineq}
	&\sup_{t\in[0,T]}\ir \Big\{a\big(u-1-\log(u)\big) + \big(\theta-1-\log(\theta)\big) + \frac{v^2}{2}\Big\} \,{\rm d}x \nonumber\\
	+	\,&\int_0^T\ir  \Big\{\frac{\mu v_x^2}{u\theta} + \frac{\kappa \theta_x^2}{u\theta^2} \Big\} \, {\rm d}x{\rm d}t \leq qE_0.
	\end{align}
\end{proposition}

\begin{proof}
	First  we derive an alternative version of the evolution equation for temperature: by substituting the mass and momentum equations \eqref{eq_mass equation}\eqref{eq_momentum equation} into Eq. \eqref{eq_temperature equation}, one obtains:	\begin{equation}\label{eq_simplified temperature equation}
	\theta_t + a\frac{\theta}{u}v_x = \big(\kappa\frac{\theta_x}{u}\big)_x + \mu\frac{v_x^2}{u} + qK\phi(\theta)Z.
	\end{equation}
	
	Now let us multiply $a(1-\frac{1}{u})$ to Eq. \eqref{eq_mass equation}, $v$ to Eq. \eqref{eq_momentum equation} and $(1-\frac{1}{\theta})$ to Eq. \eqref{eq_simplified temperature equation}: Adding up the resulting expressions together, we deduce that
	\begin{align}\label{dt}
	\frac{\partial}{\partial t} \big[a\big(u-1-\log(u)\big) &+ \big(\theta-1-\log(\theta)\big) + \frac{v^2}{2}\big] + \frac{\mu v_x^2}{u\theta} + \frac{\kappa \theta_x^2}{u\theta^2} \nonumber\\
	&= (1-\frac{1}{\theta})qK\phi(\theta)Z + \frac{\partial}{\partial x}\big[\frac{\mu v v_x - av\theta}{u} + (1-\frac{1}{\theta})\frac{\kappa \theta_x}{u}+av \big].
	\end{align}

Then, for the right-hand side, we observe that $\ir \frac{\partial}{\partial x}\big[\frac{\mu vv_x - av\theta}{u} + (1-\frac{1}{\theta})\frac{\kappa \theta_x}{u}+av \big]\dx=0$ holds due to the far-field condition \eqref{eq_far-field condition}. In light of Proposition \ref{proposition_energy and entropy} we then have
	\begin{equation*}
	\int_0^T\ir (1-\frac{1}{\theta})qK\phi(\theta)Z\,{\rm d}x{\rm d}t \leq q\ir Z_0(x)\,{\rm d}x \leq qE_0,
	\end{equation*}
which completes the proof once Eq. \eqref{dt} is integrated over $[0,T]\times\rea$.
	
\end{proof}

\section{Uniform Bounds for the Specific Volume $u$}

	In this section we establish the uniform (in space-time) upper and lower bounds for $u$. The proof is an adaptation of  the classical argument by Kazhikhov and coauthors, {\it cf.} \cite{K82,KS77} and the references cited therein. It relies on an explicit representation formula for $u$ in terms of the other dynamical variables, which are in turn controlled by the entropy formula, {\it i.e.},  Eq. \eqref{eq_entropy ineq}. 
	
	Before stating and proving further results, let us first explain the notations and conventions adopted in the rest of the paper: 
\begin{itemize}
\item 
We use $C_i$, $i \in\{0,1,2,3,\ldots\}$, to denote the positive constants depending only on the initial data and the fluid. More precisely, 
\begin{equation*}
0 < C_i = C_i \Big(a, \mu, \kappa, q, K, d, \phi(\cdot), \|(u_0-1, v_0, \theta_0 - 1, Z_0)\|_{H^1(\rea)}, \inf_{\rea}\phi, \sup_{\rea} \phi, \inf_{\rea}\theta_0\Big).
\end{equation*}
It is crucial 
that $C_i$'s are independent of the uniform norm of $\phi'$.

\item
We denote by $\epsilon$ the generic small constants appear in the estimates. They only depend on the constants of the fluid, unless otherwise specified.
\item
In \S\S 3-6 we assume momentarily that the reacting rate function $\phi(\theta)$ is $C^1$, and the following bounds are valid:
	\begin{equation}
	(\clubsuit)\quad\begin{cases}
	\|\phi\|_{C^1(\rea)} \leq \delta ^{-1} < \infty;\\
	0 \leq \phi(\theta) \leq M < \infty;\\
	\phi(\theta) = 0 \text{ whenever } \theta \leq \theta_{\text{ignite}}.
	\end{cases}
\end{equation}
It is crucial for $M$ to be independent of $\delta$, which enables us to apply $(\clubsuit)$ to the mollified versions of $\phi$ and derive uniform estimates. As a consequence, in \S 6 we can pass to the limits to recover the estimates for discontinuous $\phi$. 
\end{itemize}

The main theorem in this section is as follows:

\begin{theorem}\label{thm_uniform upper and lower bds for v}
	Let $(u,v,\theta, Z)$ be a weak solution to the system \eqref{eq_mass equation}-\eqref{eq_conditions on id} on $[0,T]\times\rea$. Then, there exists a universal constant $C_0$ such that
	\begin{equation}
	0<C_0^{-1} \leq u(\cdot,\cdot) \leq C_0 < \infty \text{ on } [0,T] \times \rea.
	\end{equation}
\end{theorem}

A key ingredient of the proof  is the following lemma, which is a modification of the now-standard ``localisation trick'' in \cite{K82,KS77}:

\begin{lemma}\label{lemma Modified Localisation Trick}
There exists two universal constants $\gamma_1, \gamma_2$ such that
\begin{equation}\label{xxx}
0< \gamma_1 \leq \int_{I_k} u(t,x)\,\dd x, \int_{I_k} \theta(t,x)\,\dd x \leq \gamma_2 < \infty
\end{equation}
for all $k \in \mathbb{Z}$ and $t>0$; here $I_k=[k,k+1]$. Moreover, given any such $t$ and $k$, we can find $b_k(t)\in I_k$ so that 
\begin{equation}\label{ineq}
0<\gamma_1 \leq u\big(t,b_k(t)\big), \theta\big(t, b_k(t)\big) \leq \gamma_2 <\infty.
\end{equation}
\end{lemma}

\begin{proof}[Proof for Lemma \ref{lemma Modified Localisation Trick}]

Let us denote by
\begin{equation}
\psi (s):=  s-1-\log(s),
\end{equation}
which is a convex function on $[0,\infty)$. Then, on each space interval $I_{k}=[k,k+1]$, $k \in \mathbb{Z}$, applying the entropy formula \eqref{eq_entropy ineq} and Jensen's inequality we deduce that
	\begin{equation*}
	\begin{cases}
	\psi\Big(\int_{I_k} u(t,x)\,{\rm d}x\Big) \leq \int_{I_k} \psi(u(t,x))\,{\rm d}x \leq C(E_0)\\
	 \psi\Big(\int_{I_k} \theta(t,x)\,{\rm d}x\Big) \leq \int_{I_k} \psi(\theta(t,x))\,{\rm d}x \leq C(E_0).
	\end{cases}
	\end{equation*}
Moreover, as $\psi$ ({\it i.e.} is monotonically decreasing  from infinity to zero on $(0,1]$ and monotonically increasing  from zero to infinity on $[1,\infty)$), we can find two positive constants $\gamma_1, \gamma_2$ such that, for all $k \in \mathbb{Z}, t>0$, 
\begin{equation*}
0< \gamma_1 \leq \int_{I_k} u(t,x)\,{\rm d}x, \int_{I_k} \theta(t,x)\,{\rm d}x \leq \gamma_2 < \infty.
\end{equation*}
This prove the first part of the  lemma. 

For the second part, we fix a small constant $\epsilon \in (0,1/2)$. Then, we take any $t>0$ and consider the ``exceptional'' set:
\begin{equation}
\mathfrak{S}_k(t):=\{ x\in I_k: \theta(t,x)<\gamma_1 \text{ or }  \theta(t,x)>\gamma_2 \text{ or }  u(t,x)<\gamma_1 \text{ or }  u(t,x)>\gamma_2 \}.
\end{equation}

By investigating the graph of $\psi$ we note the following: On $\mathfrak{S}_k(t)$, either $\psi(\theta)$ or $a\psi(u)$ is greater than some large number $\tilde{K} = \tilde{K} (\gamma_1, \gamma_2) \geq 1$.  Thus, employing Eq. \eqref{eq_entropy ineq} and the Chebyshev's  inequality, we deduce that
\begin{equation}
\tilde{K} \Big|\mathfrak{S}_k(t)\Big| \leq \sup_{0 \leq t < T} \ir [a\psi(u) + \psi(\theta)]\, {\rm d}x \leq qE_0,
\end{equation}
where for a Borel set $B\subset\rea$ its one-dimensional Lebesgue measure is denoted as $|B|$. 

Now, we observe that  $\tilde{K}$ increases if either $\gamma_2$ increases or $\gamma_1$ decreases. Hence, by suitably choosing $\gamma_1, \gamma_2$ which depend only on $a, q, E_0$, we obtain the bound:
\begin{equation}
|\mathfrak{S}_k(t)| \leq 1-\epsilon
\end{equation}
uniformly in time. Therefore, for each $t\in[0,T)$, we pick an arbitrary $b_k(t) \in I_k \setminus \mathfrak{S}_k(t)$ to complete the proof.

\end{proof}

With Lemma \ref{lemma Modified Localisation Trick},  we are at the stage of proving our main theorem in this section. The proof is a straightforward adaptation of the estimates in \cite{J99}\cite{J02} by S. Jiang. In fact,  similar estimates have been obtained in \cite{CHT03}\cite{K82}\cite{SK81} and several other works, but not uniformly in time. The crucial observation in \cite{J99}\cite{J02} is that, although $\int_0^t \frac{\theta(\tau,x)}{u(\tau,x)} \,{\rm d}\tau$ is difficult to be bounded even at a single point $x=b_k(t)$, its {\em spatial average} $\int_s^t\int_{I_k} \frac{\theta(\tau,\xi)}{u(\tau,\xi)} \,{\rm d}\xi {\rm d}\tau$ can nevertheless be controlled (here $0\leq s<t\leq T$).

 Throughout the following proof, let us write $N$ for universal constants independent of $t,x,k$. In particular, the independence of $k$ will be justified at the end of the argument.

\begin{proof}[Proof for Theorem \ref{thm_uniform upper and lower bds for v}]

The proof is divided into three steps:

\smallskip
{\bf 1.}
	First, we choose a spatial cut-off function $\chi\in C^\infty_c([0,\infty)), \chi \equiv 1$ on $[0,k]$, $\chi\equiv 0$ on $(k+1, \infty)$ and $0 \leq \|\chi\|_{C^1} \leq 1$. Testing against the momentum equation \eqref{eq_momentum equation}, one obtains: 
	\begin{equation}\label{eq for eff viscous flux}
	-\int_x^\infty \big[v(t,\xi)\chi(\xi)\big]_t {\rm d}y = \sigma(t, x) + \int_{I_k} \chi_x(\xi)\sigma(t,\xi){\rm d}\xi \qquad \text{ for all } x \in I_k.
	\end{equation}
Here, $\sigma$ is the {\it effective viscous flux}, defined as
\begin{equation}
\sigma := \frac{\mu v_x - a\theta}{u}.
\end{equation}

Starting with Eq. \eqref{eq for eff viscous flux}, an integration over $[0,t]$ gives us:
\begin{align*}
&\int_x^\infty \big(v(t,\xi)-v_0(\xi) \big)\chi(\xi)\, {\rm d}\xi \\
 =\,& \mu \log\frac{u(t,x)}{u_0(x)} -a\int_0^t\frac{\theta}{u}(\tau,x)\,{\rm d}\tau + \int_0^t\int_{I_k} \chi_x(\xi)\sigma(\tau,\xi)\,{\rm d}\xi {\rm d}\tau.
\end{align*}
Then, we take the exponential of both sides to derive that
\begin{equation}
u(t,x) = u_0(x) \times \frac{\exp\Big\{\frac{1}{\mu}\int_x^\infty \Big[v(t,\xi)-v_0(\xi)\Big]\chi(\xi)\,{\rm d}\xi \Big\}\exp\Big\{\frac{a}{\mu}\int_0^t\frac{\theta(\tau,x)}{u(\tau,x)}\,{\rm d}\tau\Big\}}{\exp \Big\{\frac{1}{\mu}\int_0^t \int_{I_k}\chi_x(\xi)\sigma(\tau,\xi)\,{\rm d}\xi {\rm d}\tau\Big\}}.
\end{equation}
Now, introduce the following short-hand notations in the above expression:
\begin{equation}
\begin{cases}
B(t,x):= v_0(x) \exp\Big\{\frac{1}{\mu}\int_x^\infty \Big(v_0(\xi)-v(t,\xi)\Big)\chi(\xi)\,{\rm d}\xi \Big\}, \\
Y(t):= \exp \Big\{\int_0^t\int_{I_k}\chi_x(\xi)\sigma(\tau,\xi)\,{\rm d}\xi {\rm d}\tau\Big\}= \exp \Big\{\int_0^t\int_{I_k}\frac{\mu v_x(\tau,\xi) - a\theta(\tau,\xi)}{u(\tau,\xi)}\chi_x(\xi)\,{\rm d}\xi {\rm d}\tau\Big\}.
\end{cases}
\end{equation}
Thus we have 
\begin{equation*}
\frac{1}{Y(t)B(t,x)}=\frac{1}{u(t,x)}\exp\Big\{\frac{a}{\mu}\int_0^t\frac{\theta(\tau,x)}{u(\tau,x)}{\rm d}\tau\Big\}.
\end{equation*}
We multiply the above equation by $a\mu^{-1} \theta(t,x)$ and integrate over $t$ to obtain:
\begin{equation*}
\exp \Big\{\frac{a}{\mu}\int_0^t\frac{\theta(\tau,x)}{u(\tau,x)}{\rm d}\tau \Big\}= 1+ \frac{a}{\mu}\int_0^t \frac{\theta(\tau,x)}{Y(\tau)B(\tau, x)}{\rm d}\tau.
\end{equation*}
This leads to an explicit representation formula for the specific volume, namely 
\begin{equation}\label{eq_rep formula for v}
u(t,x) = Y(t)B(t,x) + a\mu^{-1} \int_0^t\frac{Y(t)B(t,x)\theta(t,x)}{Y(\tau)B(\tau,x)}{\rm d}\tau.
\end{equation}

\smallskip
{\bf 2.}
In this step we derive uniform bounds for $u$ based on the above representation formula. First, by $\sup_{0\leq t\leq T}\ir v^2(t,x) \dx\leq 2qE_0$ (which is an immediate consequence of the entropy formula, {\it i.e.}, Eq. \eqref{eq_entropy ineq}), one  concludes that
\begin{equation}\label{eq_bound for B}
0<N^{-1}\leq B(t,x)\leq N <\infty.
\end{equation}

Next, for any $0<s<t\leq T$,  a lower bound can be derived for $\int_s^t \theta(\tau,x) {\rm d}\tau$ on $I_k$ {\it uniformly in $k$}. For this purpose, we first employ Jensen's inequality to estimate
\begin{align}
\int_s^t \theta(\tau,x){\rm d}\tau &\geq (t-s) \exp\Big\{\int_s^t \frac{1}{t-s}\log(\theta) {\rm d}\tau\Big\}\nonumber\\
&= (t-s) \exp \Big\{\frac{1}{t-s} \int_s^t \Big[ \int_{b_{k}(t)}^x \frac{\theta_x (\tau,y)}{\theta(\tau,y)}{\rm d}y + \log \theta(\tau,b_k(t)) \Big] {\rm d}\tau\Big\}\nonumber\\
&\geq (t-s)\exp \Big\{N - \frac{1}{t-s}\Big|\int_s^t\int_{b_k(t)}^x \frac{\theta_x}{\theta}(\tau,y) {\rm d}y {\rm d}\tau\Big| \Big\}\nonumber\\
&\geq N(t-s) e^{-\frac{1}{N(t-s)}},
\end{align}
which holds in view of the inequalities $\int_0^T\ir \frac{\kappa \theta_x^2}{\theta^2} \,{\rm d}x{\rm d}t\leq qE_0$ (see Eq. \eqref{eq_entropy ineq}), $\int_{I_k} u(t,x){\rm d}x \leq \gamma_2$ (due to Lemma \ref{lemma Modified Localisation Trick}), and the concavity of $\log$. Then, we have:
\begin{align}
\int_s^t\int_{I_k} \sigma(\tau,x)\,{\rm d}x{\rm d}\tau & \leq  (-a + \epsilon) \int_s^t\int_{I_k} \frac{\theta(\tau,\xi)}{u(\tau,\xi)}\, {\rm d}\xi {\rm d}\tau 
+ C(\epsilon)\int_s^t\int_{I_k}\frac{\mu u_x^2}{u\theta}\, {\rm d}\xi {\rm d}\tau \nonumber\\
&\leq N - \frac{a}{2} \int_s^t\int_{I_k} \frac{\theta(\tau,\xi)}{u(\tau,\xi)}\, {\rm d}\xi {\rm d}\tau \nonumber\\
&\leq N - N\int_s^t \inf_{I_k} \theta(\tau,\cdot) \Big(\int_{I_k}\frac{1}{u(\tau,\xi)}{\rm d}\xi\Big)\, {\rm d}\tau \nonumber\\
&\leq N-N\int_s^t \inf_{I_k} \theta(\tau,\cdot) {\Big(\int_{I_k}u(\tau,\xi){\rm d}\xi\Big)}^{-1}\, {\rm d}\tau \leq N-N^{-1}(t-s),
\end{align}
for which one utilises Jensen's inequality, the lower bound on $\int_s^t\theta {\rm d}\tau$, Eq. \eqref{eq_entropy ineq}, as well as Lemma \ref{lemma Modified Localisation Trick}. Hence, for arbitrary $0 \leq \tau \leq t$ the following holds:
\begin{equation}\label{eq_bound for Y}
0\leq Y(t)\leq Ne^{-\frac{t}{N}},\qquad \frac{Y(t)}{Y(\tau)} \leq Ne^{-\frac{t-\tau}{N}}.
\end{equation}

\smallskip
{\bf 3.} Using the bounds in Eqs. \eqref{eq_bound for B} and \eqref{eq_bound for Y}, the representation formula \eqref{eq_rep formula for v} and the localisation trick (Lemma \ref{lemma Modified Localisation Trick}), we now  conclude that
\begin{equation}\label{inequality}
\begin{cases}
u(t,x) \leq N+ N\int_0^t \theta(\tau,x) e^{-(t-\tau)/N}{\rm d}\tau,\\
\gamma_1 \leq \int_{I_k} u(t,x){\rm d}x \leq Ne^{-t/N}+N\int_0^t{Y(t)}/{Y(\tau)}{\rm d}\tau \qquad \text{ on } [0,\infty) \times I_k.
\end{cases}
\end{equation}
On the other hand, we have a reverse inequality which bounds $\theta$ in terms of $u$:
\begin{align}\label{eq_reverse ineq}
\big|\sqrt{\theta}(t,x) - \sqrt{\theta}(t,b_k(t))\big| &\leq \int_{I_k} \frac{|\theta_x(t,x)|}{\sqrt{\theta}(t,x)}{\rm d}x\nonumber\\
&\leq \Big(\int_{I_k}\frac{\theta_x^2}{u\theta}(t,x) {\rm d}x\Big)^{\frac{1}{2}} \Big(\int_{I_k} u(t,x)\theta(t,x){\rm d}x\Big)^{\frac{1}{2}}\nonumber\\
&\leq  \sqrt{\gamma_2}\Big(\int_{I_k}\frac{\theta_x^2}{u\theta}(t,x) {\rm d}x\Big)^{\frac{1}{2}} \max_{I_k}\sqrt{\theta(t,\cdot)} \qquad \text{ on } [0,\infty) \times I_k,
\end{align}
again due to Lemma \ref{lemma Modified Localisation Trick}.

Finally, in view of Eq. \eqref{eq_entropy ineq}, an application of Gr\"{o}nwall lemma to Eqs. \eqref{eq_reverse ineq} and \eqref{inequality} gives us the uniform-in-time upper and lower bounds for $u$. In particular, the constants $N$ are independent of $k$. Thus the proof is now complete.

\end{proof}

\section{The Crucial Estimate for $v$ and $\theta$}

In this section we establish
a key estimate involving $v, \theta, v_x, \theta_x$ and suitable powers of them. This inequality is an adaptation of the key estimate in \cite{LL14} ({\it cf.} Lemma 2.2 therein). However, due to  the presence of the chemical reaction processes, extra work needs to be done in order to control the variable $Z$. 

Our main result in the current section is summarised as follows:

\begin{theorem}\label{thm_crucial estiamte in theta and u, Li Liang type}
	Let $(u,v,\theta, Z)$ be a weak solution to the system \eqref{eq_mass equation}-\eqref{eq_conditions on id} on $[0,T]\times\rea$. Then there exists $C_1>0$, depending only on the initial data, such that
	\begin{equation}
	\sup_{0\leq t\leq T} \ir \big[(\theta-2)_{+}^2 + v^4\big](t,x) \,{\rm d}x + \int_0^T \ir \big[(1+\theta+v^2) v_x^2 + \theta_x^2 \big](t,x) \,{\rm d}x {\rm d}t \leq C_1.
	\end{equation}

\end{theorem}

To simplify the presentation, let us collect several simple algebraic identities, which are to be repetitively invoked in the subsequent development:
\begin{lemma}\label{lemma_identity for theta}
Let us denote the spatial level sets by  
\begin{equation}
\Sigma_a(t):= \{x\in\rea: \theta(t,x) \geq a\},
\end{equation}
and write $\psi(s)=s-1-\log(s)$ on $\rea_{+}$ as before. Then,
\begin{enumerate}
\item[{\em (1)}]
For any $a>1$, there exists a universal constant $C=C(a)$, such that
\begin{equation}\label{algebraic identity 0}
\sup_{0\leq t<\infty} \int_{\Sigma_a(t)} \theta(t,x)\,{\rm d}x \leq C(a) \sup_{0\leq t <\infty} \ir \psi\big(\theta(t,x)\big)\, {\rm d}x \leq C(a)qE_0.
\end{equation} 

\item[{\em (2)}]
For $a>1$ there exists $C=C(a)$ such that
\begin{equation}\label{estimate on theta-1 square}
\sup_{0\leq t<\infty}\int_{\rea \setminus \Sigma_a(t)} \big(\theta(t,x)-1\big)^2 \, {\rm d}x \leq C(a) \sup_{0\leq t<\infty}\ir \psi \circ \theta(t,x)\,{\rm d}x \leq C(a)qE_0;
\end{equation}
\item[{\em (3)}]
We have the algebraic inequalities (where $B>0$ is a constant)
\begin{equation}\label{algebraic identity}
\begin{cases}
\theta^2\chi_{\Sigma_2(t)} \leq 16(\theta-3/2)^2_{+},\\
\theta(\theta-2)_{+} \leq 2(\theta-\frac{3}{2})_{+}^2,\\
(\theta-1)^2\chi_{\Sigma_2(t)}\leq B(\theta-\frac{3}{2})_{+}^2.
\end{cases}
\end{equation}
\item[{\em (4)}]
For any $\psi \in H^1(\rea)=W^{1,2}(\rea)$, we have
\begin{equation}\label{simple identity}
\sup_{x\in\rea} |\psi(x)|^2 \leq \|\psi'\|_{L^2(\rea)} \|\psi\|_{L^2(\rea)}\leq \|\psi\|_{H^1(\rea)}\|\psi\|_{L^2(\rea)}.
\end{equation}

\end{enumerate}
\end{lemma}
\begin{proof}

(1)--(3) follow from straightforward algebraic computations; we omit the details here. Let us only comment that in (1), the following choice of constant
\begin{equation*}
C(a)=\frac{a}{\psi(a)}=\frac{a}{a-1-\log(a)}
\end{equation*}
satisfies the requirement, as $\psi(s)$ has a double zero at $1$; also, in (3) any $B>\frac{4}{3}$ works. Finally, (4) is the standard Sobolev inequality corresponding to the embedding $H^1(\rea)\hookrightarrow C^0(\rea)$.

\end{proof}

\begin{proof}[Proof for Theorem \eqref{thm_crucial estiamte in theta and u, Li Liang type}]
	
We divide our arguments in four steps.

\smallskip
{\bf 1.} We start by deriving an energy estimate for the  temperature equation, in the form of Eq. \eqref{eq_simplified temperature equation}. The aim is to bound the $L^2$ norm of $\theta$ in the ``high-temperature region'', in terms of other dynamical variables.

For this purpose let us multiply $(\theta-2)_{+}$ to  Eq. \eqref{eq_simplified temperature equation}. This gives us
	\begin{align}\label{by parts 1}
	&(\theta-2)_{+}\theta_t + \kappa \frac{\theta_x}{u}\Big[(\theta-2)_{+}\Big]_x \nonumber\\
	=\quad & \frac{1}{2} \Big[(\theta-2)_{+}^2\Big]_t + \kappa\frac{1}{2} \big(\frac{\theta_x^2}{u}\big)_{x} \chi_{\Omega_2(t)}\nonumber\\
	=\quad & \Big[\frac{\kappa (\theta-2)_{+}\theta_x}{u}\Big]_x +\mu \frac{v_x^2(\theta-2)_{+}}{u} - a\frac{\theta}{u}u_x(\theta-2)_{+} + qK\phi(\theta)Z (\theta-2)_{+}.
	\end{align}
Noticing that $\big(\theta(t,x)-2\big)_{+} \rightarrow 0$ as $|x|\rightarrow \infty$, we  integrate over $[0,T]\times\rea$ to derive:
	\begin{align*}
	&\frac{1}{2} \ir \big[(\theta(T,x)-2)_{+}\big]^2 \dx + \int_0^T\int_{\Sigma_2(t)}\kappa\frac{\theta_x^2}{u} \,{\rm d}x{\rm d}t \\
	=\: & \frac{1}{2}\ir \big[(\theta_0(x)-2)_{+}\big]^2 \dx+ \int_0^T\ir \mu\frac{v_x^2}{u} (\theta-2)_{+} \,{\rm d}x{\rm d}t \\
	-\:&a\int_0^T\ir \frac{\theta v_x}{u}(\theta-2)_{+} \,{\rm d}x{\rm d}t + \int_0^T\ir qK\phi(\theta)Z(\theta-2)_{+} \,{\rm d}x{\rm d}t.
	\end{align*}
	
	On the other hand, multiplying $2v(\theta-2)_{+}$ to the momentum equation \eqref{eq_momentum equation} yields that
	\begin{align}\label{by parts 2}
	\Big[v^2(\theta-2)_{+}\Big]_t + 2\mu \frac{v_x^2}{u} &(\theta-2)_{+} =  2a\frac{\theta}{u}v_x(\theta-2)_{+} + 2a \frac{v\theta}{u} \Big[(\theta-2)_{+} \Big]_x -2\mu \frac{v_x^2}{u}(\theta-2)_{+}\nonumber\\
	-&2\mu \frac{vv_x}{u}\Big[(\theta-2)_{+}\Big]_x +2\Big[\frac{\mu v v_x(\theta-2)_{+} - av\theta(\theta-2)_{+}}{u}\Big]_x.
	\end{align}
	
Hence, integrating over $[0,T]\times \rea$, we obtain as follows:
	\begin{align*}
	&\ir \Big\{v^2(\theta-2)_{+}(T,x)\Big\} \,{\rm d}x + \int_0^T\ir \frac{2\mu(\theta-2)_{+}}{u}v_x^2  \,{\rm d}x{\rm d}t \\
	= & \ir \Big\{v_0^2(x)\big(\theta_0(x)-2\big)_{+}\Big\} \,{\rm d}x + 2a \int_0^T\ir \frac{\theta v_x}{u}(\theta-2)_{+} \,{\rm d}x{\rm d}t +2a\int_0^T\int_{\Sigma_2(t)} \frac{v\theta\theta_x}{u}\,  {\rm d}x{\rm d}t \\
-&2\mu \int_0^T \int_{\Sigma_2(t)}  \frac{vv_x}{u}\theta_x  \,{\rm d}x{\rm d}t + \int_0^T\int_{\Sigma_{2}(t)} v^2 \big[ (\theta-2)_{+}\big]_t \,{\rm d}x{\rm d}t.
	\end{align*}

Adding the above two integral expressions together,  evaluating $[(\theta-2)_{+}]_t$ on the level set $\Sigma_2(t)$, and employing the evolution equation \eqref{eq_simplified temperature equation} for $\theta$, we now arrive at:
\begin{align}\label{eq_estimating I1-I6}
&\frac{1}{2} \ir \Big\{ (\theta-2)_{+}^2 + v^2(\theta-2)_{+}\Big\}(T,x)\,{\rm d}x +\mu \int_0^T\ir  \frac{(\theta-2)_{+}}{u}v_x^2 \, {\rm d}x{\rm d}t + \kappa\int_0^T\int_{\Sigma_2(t)}  \frac{\theta_x^2}{u} \,{\rm d}x{\rm d}t\nonumber\\
=\:&\frac{1}{2}\ir \Big\{(\theta_0-2)_{+}^2 + v_0^2 (\theta_0 -2)\big\}(x)\,{\rm d}x + a\int_0^T\ir \Big\{\frac{\theta}{u} v_x(\theta-2)_{+} \Big\}\,{\rm d}x{\rm d}t \nonumber\\
&+ 2a\int_0^T\int_{\Sigma_2(t)}  \frac{v\theta \theta_x}{u} \,{\rm d}x {\rm d}t -2\mu \int_0^T \int_{\Sigma_2(t)}  \frac{vv_x}{u}\theta_x  \, {\rm d}x{\rm d}t + \int_0^T \int_{\Sigma_2(t)}  v^2 \big(\frac{\kappa\theta_x}{u}\big)_x \, {\rm d}x {\rm d}t \nonumber\\
&+ \int_0^T\int_{\Sigma_2(t)} \Big\{v^2 \frac{\mu v_x^2 - a\theta v_x}{u}\Big\} \, {\rm d}x{\rm d}t + \int_0^T\ir \Big\{qK\phi(\theta)Z\Big[(\theta-2)_{+}+v^2 \Big]\Big\} \,{\rm d}x{\rm d}t\nonumber\\
=:\:&\frac{1}{2}\ir \Big\{(\theta_0-2)_{+}^2 + v_0^2 (\theta_0 -2)\Big\}(x) \,{\rm d}x + \sum_{j=1}^6 I_j.
\end{align}

	\smallskip
	{\bf 2.} Now, our task is to estimate $I_1, I_2, \ldots, I_6$ term by term. To this end, we use Young's inequality (or Cauchy-Schwarz in the simplest form) repeatedly to separate each $I_j$ into one ``small'' and one ``large'' part: the ``small'' part can be absorbed by the left-hand sides, and the ``large'' part can be controlled via the uniform bounds established in \S 2, together with the uniform boundedness of $u$ ({\it cf.} Theorem \ref{thm_uniform upper and lower bds for v}).
	
\begin{itemize}
\item
	For $I_1$, using Eqs.  \eqref{algebraic identity 0}\eqref{algebraic identity}, we estimate as follows:
	\begin{align}\label{estimate for I1}
	|I_1| &\leq \epsilon_1 \int_0^T \ir \frac{(\theta-2)_{+}}{u}v_x^2 \, {\rm d}x{\rm d}t + C(\epsilon_1) \int_0^T\ir \theta^2(\theta-2)_{+} \,{\rm d}x{\rm d}t \nonumber\\
&\leq \epsilon_1 \int_0^T \ir \frac{(\theta-2)_{+}}{u}v_x^2 \, {\rm d}x{\rm d}t + C(\epsilon_1) \int_0^T \sup_{\rea} \big[(\theta(t, \cdot)-\frac{3}{2})_{+}\big]^2 \,{\rm d}t.
	\end{align}

\item
For $I_2$, notice that
\begin{align*}
|I_2| \leq \epsilon_2 \int_0^T\int_{\Sigma_2(t)}\frac{\theta_x^2}{u} {\rm d}x{\rm d}t + C(\epsilon_2) \int_0^T\int_{\Sigma_2(t)} v^2\theta^2 {\rm d}x{\rm d}t.
\end{align*}
Again, we use Eq. \eqref{algebraic identity}  and Eq. \eqref{eq_entropy ineq} to derive that
\begin{equation}\label{estimate_I2}
|I_2| \leq \epsilon_2 \int_0^T\int_{\Sigma_2(t)}\frac{\theta_x^2}{u} \,{\rm d}x{\rm d}t +C(\epsilon_2) \int_0^T \sup_\rea \big[(\theta(t,\cdot)-\frac{3}{2})_{+}\big]^2\, {\rm d}t.
\end{equation}

\item
For $I_3$, let us directly bound
\begin{equation}\label{estimate_I3}
|I_3| \leq \epsilon_3 \int_0^T\int_{\Sigma_2(t)} \frac{\theta_x^2}{u} \, {\rm d}x{\rm d}t +  C(\epsilon_3) \int_0^T \int_{\Sigma_2(t)} v^2v_x^2 \,{\rm d}x{\rm d}t.
\end{equation}
\item
$I_4:= \int_0^T\int_{\Sigma_2(t)} v^2 \big(\frac{\kappa \theta_x}{u}\big)_x\, {\rm d}x{\rm d}t$ is a term with special structure. By a standard trick, we integrate against a test function $\varphi(\theta)$:
\begin{align*}
\int_0^T\int_{\Sigma_2(t)}& v^2 \varphi(\theta) \Big[\frac{\kappa\theta_x}{u}\Big]_x \,{\rm d}x{\rm d}t  = \int_0^T\int_{\Sigma_2(t)} \Big[\frac{\kappa v^2\varphi(\theta)\theta_x}{u}\Big]_x \,{\rm d}x{\rm d}t\\
& - 2\kappa \int_0^T\int_{\Sigma_2(t)} \frac{vv_x\theta_x}{u}\varphi(\theta) \,{\rm d}x{\rm d}t - \kappa \int_0^T\int_{\Sigma_2(t)}\frac{v^2\varphi'(\theta)\theta_x^2}{u} \,{\rm d}x{\rm d}t.
\end{align*}
Hence, choosing a sequence of test functions $\varphi_\eta \in C^\infty[0,\infty)$ such that $\varphi_\eta(\theta) \equiv 0$ for $\theta \leq 2$, $\varphi_\eta(\theta) \equiv 1$ for $\theta \geq 2+\eta$, and $\varphi_\eta'(\theta)\geq 0$, we immediately get:
\begin{align}\label{estimate for I4}
I_4 &= \lim_{\eta \searrow 0} \int_0^T\int_{\Sigma_2(t)} v^2 \varphi_\eta(\theta) \Big[\frac{\kappa\theta_x}{u}\Big]_x \,{\rm d}x{\rm d}t \nonumber = \lim_{\eta \searrow 0} - 2\kappa \int_0^T\int_{\Sigma_2(t)} \frac{vv_x\theta_x}{u}\varphi_\eta(\theta)\, {\rm d}x{\rm d}t \nonumber\\
&\leq \epsilon_4  \int_0^T\int_{\Sigma_2(t)} \frac{\theta_x^2}{u}\,{\rm d}x{\rm d}t + C(\epsilon_4) \int_0^T\int_{\Sigma_2(t)}  v^2v_x^2\, {\rm d}x{\rm d}t.
\end{align}
\item
$I_5$ is simple: by Eqs.  \eqref{eq_entropy ineq} and \eqref{algebraic identity},
\begin{align}\label{estimate for I5}
|I_5| &\leq C \int_0^T\int_{\Sigma_2(t)} \big(v^2 v_x^2 + v^2\theta^2\big) \, {\rm d}x{\rm d}t \nonumber\\
&\leq C\int_0^T\int_{\Sigma_2(t)} v^2 v_x^2 \,{\rm d}x{\rm d}t + C\int_0^T \sup_\rea \big\{\big(\theta(t,\cdot)-\frac{3}{2}\big)_{+}\big\}^2 \,{\rm d}t.
\end{align}
\item
Finally let us deal with $I_6$, which is the term involving $Z$. In view of the boundedness of $\phi$ in the $C^0$-topology, $\sup_{0\leq t \leq T}\int_\rea Z (t,x)\dx \leq C$ ({\it cf}. Proposition \ref{proposition_energy and entropy}), and that $(\theta-2)_{+} \leq (\theta-3/2)_{+}^2$ ({\it cf.} Lemma \ref{lemma_identity for theta}), we achieve at the following:
\begin{equation}\label{estimate for I6}
\Big|\int_0^T\ir qK\phi(\theta)Z(\theta-2)_{+}\,{\rm d}x{\rm d}t\Big| \leq C \int_0^T \sup_\rea \big\{\big(\theta(t,\cdot)-\frac{3}{2}\big)_{+}\big\}^2 \,{\rm d}t.
\end{equation}

On the other hand, by Eq. \eqref{eq_entropy ineq} and the identity in Eq. \eqref{simple identity}, we have:
\begin{align}
\Big|\int_0^T \ir qK\phi(\theta) Z v^2\,{\rm d}x{\rm d}t\Big|  \leq \:& C\int_0^T \|v(t,\cdot)\|_{L^2(\rea)}\|v_x(t,\cdot)\|_{L^2(\rea)} \,{\rm d}t\nonumber\\
\leq \;& C \int_0^T\ir v_x^2 \,{\rm d}x{\rm d}t.
\end{align}
Thus
\begin{equation}
|I_6| \leq C\int_0^T\int_{\Sigma_2(t)} v_x^2 \,{\rm d}x{\rm d}t + C\int_0^T \sup_\rea \big\{\big(\theta(t,\cdot)-\frac{3}{2}\big)_{+}\big\}^2 \, {\rm d}t.
\end{equation}

\end{itemize}

	Now, we combine the previous estimates in Eqs.  \eqref{estimate for I1}--\eqref{estimate for I6} to control the right-hand side of Eq. \eqref{eq_estimating I1-I6}. Indeed, selecting $\epsilon_1 = \frac{1}{2}\mu$ and $\epsilon_2=\epsilon_3=\epsilon_4 = \frac{1}{4}\kappa$ proves the existence of a universal constant $C_2>0$, depending only on the initial data, $\mu,\kappa,a,q,K,\|\phi\|_{L^\infty}$ and $C_0$ in Theorem \ref{thm_uniform upper and lower bds for v}, so that for all $0 \leq t \leq T$ the following holds:
	\begin{align}\label{zzz}
	&\frac{1}{2} \ir \big\{(\theta-2)_{+}^2 + v^2(\theta-2)_{+}\big\}(t,x) \,{\rm d}x +\mu \int_0^T\ir \frac{\theta}{u}v_x^2 \,{\rm d}x{\rm d}t + \kappa\int_0^T\int_{\Sigma_2(t)} \frac{\theta_x^2}{u} \,{\rm d}x{\rm d}t\nonumber\\
\leq \: & C_2 + C_2 \int_0^T \sup_\rea \big\{ \big(\theta(t,\cdot)-\frac{3}{2}\big)_{+} \big\}^2 \, {\rm d}t + C_2 \int_0^T\int_{\rea}  (1+v^2)v_x^2 \, {\rm d}x{\rm d}t.
\end{align}

\smallskip
{\bf  3.} In the third step we estimate the term
\begin{equation}
\int_0^T \sup_\rea \big\{\big(\theta(t,\cdot)-\frac{3}{2}\big)_{+}\big\}^2 \, {\rm d}t
\end{equation}
on the right-hand side of Eq.  \eqref{zzz}. 

To wit, the identity \eqref{algebraic identity 0} and the entropy formula \eqref{eq_entropy ineq} imply that
\begin{align}\label{by parts add}
\int_0^T \sup_\rea \big\{\big(\theta(t,\cdot)-\frac{3}{2}\big)_{+}\big\}^2 \, {\rm d}t & \leq \int_0^T \sup_{x \in \rea} \Big\{\int_x^\infty -{\partial_x} \big(\theta(t,\xi) - 3/2\big)_{+} \, {\rm d}\xi\Big\} ^2 {\rm d}t\nonumber\\
&\leq \int_0^T\Big(\int_{\Sigma_{3/2}(t)} |\theta_x| \dx\Big)^2 \,{\rm d}t\nonumber\\
&\leq \int_0^T \Big(\int_{\Sigma_{3/2}(t)} \frac{|\theta_x|^2}{\theta} \dx\Big)\Big(\int_{\Sigma_{3/2}(t)} {\theta} dx\Big) \dt \nonumber\\
&\leq C\int_0^T \int_{\Sigma_{3/2}(t)} \frac{|\theta_x|^2}{\theta} \,{\rm d}x {\rm d}t \nonumber\\
&\leq C(\epsilon_5)\int_0^T \ir \frac{|\theta_x|^2}{u\theta^2} \,{\rm d}x{\rm d}t + \epsilon_5 \int_0^T \int_{\Sigma_{3/2}(t)} |\theta_x|^2 \,{\rm d}x{\rm d}t,
\end{align}
where Cauchy-Schwarz and the uniform boundedness of $u$ are used in the last line. Moreover, observe that the first term on right-hand side can be bounded by Eq. \eqref{eq_entropy ineq}, and by choosing $\epsilon_5 = \frac{\kappa}{4}$, the second term can be absorbed into the left-hand side of Eq. \eqref{zzz}. Thus, there is a universal constant $C_3>0$ such that for all $0\leq t \leq T$ we have:
	\begin{align}\label{eq_intermediate step}
	&\frac{1}{2}\ir \big\{(\theta-2)_{+}^2 + v^2(\theta-2)_{+}\big\}(t,x) \dx + \mu \int_0^T\ir \frac{\theta}{u}v_x^2 \,{\rm d}x{\rm d}t + \kappa \int_0^T\int_{\Sigma_2(t)}\frac{\theta_x^2}{u} \,{\rm d}x {\rm d}t \nonumber\\
	\leq\;& C_3 \times \Big(1+\int_0^T\ir (1+v^2)v_x^2 \,{\rm d}x{\rm d}t \Big).
	\end{align}

\smallskip
{\bf 4.} Finally it remains to bound the right-hand side of Eq. \eqref{eq_intermediate step}. For this purpose, we multiply  $(v^3)$ to the momentum equation \eqref{eq_momentum equation} and investigate the evolution of the $L^4$ norm of $v$, as in Kazhikhov-Shelukhin  (\cite{KS77}). In this manner we obtain:

\begin{equation}\label{by parts 3}
\frac{1}{4}\big(v^4\big)_t + 3\mu \frac{v^2 v_x^2}{u} = \Big(\frac{\mu v^3 v_x - av^3 \theta}{u}\Big)_x + 3a \frac{\theta}{u} v^2 v_x.
\end{equation}	
	
Hence, integrating over	$[0,T]\times\rea$, we find that
	\begin{align}\label{intermediate estimate 1 for u4+u2ux2}
	\sup_{0\leq t \leq T} \frac{1}{4} \ir v^4(t,x) \dx+ 3\mu  \int_0^T \ir \frac{v^2 v_x^2}{u} \,{\rm d}x{\rm d}t	= \frac{1}{4} \ir v_0^4(x)\dx +  3a \int_0^T \ir \frac{\theta}{u} v^2 v_x\, {\rm d}x{\rm d}t.
	\end{align}

To estimate the last term on the right-hand side, one makes use of the following observation in \cite{LL14}: $(u-1)$ is square-integrable due to the boundedness of $u$ and the integrability of $\psi(u)=u-1-\log u$ ({\it cf}. Theorem \ref{thm_uniform upper and lower bds for v} and Lemma \ref{lemma_identity for theta}). Hence, we consider
\begin{align*}
\int_0^T\int_{\rea} \frac{\theta}{u}v^2v_x \,{\rm d}x{\rm d}t &=  \int_0^T\int_{\rea} \frac{(1-u)v^2v_x}{u} \,{\rm d}x{\rm d}t + \int_0^T\int_{\Sigma_2(t)} \frac{(\theta-1)v^2v_x}{u} \, {\rm d}x{\rm d}t\\
&+ \int_0^T\int_{\rea\setminus\Sigma_2(t)} \frac{(\theta-1)v^2v_x}{u}  \,{\rm d}x{\rm d}t \\
 &=: K_1+K_2+K_3,
\end{align*}
and estimate $K_1, K_2, K_3$ as follows:
\begin{itemize}
\item
	For $K_1$, we bound
	\begin{align*}
	|K_1| &\leq C\int_0^T\Big\{ \big(\sup_{\rea}v^2(t,\cdot)\big) \|1-u(t,\cdot)\|_{L^2(\rea)}\|v_x(t,\cdot)\|_{L^2(\rea)} \Big\}\,{\rm d}t \\
	&\leq C\int_0^T\Big\{\|v(t,\cdot)\|_{L^2(\rea)} \|1-u(t,\cdot)\|_{L^2(\rea)}\|v_x(t,\cdot)\|^2_{L^2(\rea)}\Big\}\,{\rm d}t\\
	&\leq C\sup_{0\leq t\leq T}\|v(t,\cdot)\|_{L^2(\rea)}\sup_{0\leq t\leq T}\|1-u(t,\cdot)\|_{L^2(\rea)}\int_0^T\int_{\rea} v_x^2(t,x) \,{\rm d}x{\rm d}t\\
	&\leq C\int_0^T\int_{\rea} v_x^2(t,x) \,{\rm d}x{\rm d}t,
	\end{align*}
thanks to items (1)(4) in Lemma \ref{lemma_identity for theta} and the entropy formula, namely Eq. \eqref{eq_entropy ineq}. On the other hand, by Cauchy-Schwarz one has
\begin{align*}
\int_0^T\int_{\rea} v_x^2 \, {\rm d}x{\rm d}t &\leq \epsilon_6 \int_0^T\int_{\rea} \theta v_x^2 \,{\rm d}x{\rm d}t + C(\epsilon_6)\int_0^T\int_{\rea} \frac{v_x^2}{\theta u}\, {\rm d}x{\rm d}t,
\end{align*}
hence
\begin{equation}\label{K1}
|K_1|\leq \epsilon_6 \int_0^T\int_{\rea} \theta v_x^2 \,{\rm d}x{\rm d}t + C(\epsilon_6).
\end{equation}

\item
Similarly, to deal with $K_2$, Lemma \ref{lemma_identity for theta} gives us
\begin{equation*}
\sup_{0\leq t \leq T}\int_{\rea\setminus \Sigma_2(t)} (\theta-1)^2 {\rm d}x{\rm d}t\leq C,
\end{equation*} 
thus one readily derives 
\begin{equation}\label{K2}
|K_2|\leq \epsilon_7 \int_0^T\int_{\rea} \theta v_x^2 \,{\rm d}x{\rm d}t + C(\epsilon_7)
\end{equation}
via analogous arguments.

\item
Finally, $K_3$ is bounded as follows:
\begin{align*}
|K_3| &\leq \epsilon_8 \int_0^T\int_{\Sigma_2(t)}v^2v_x^2 \,{\rm d}x{\rm d}t + C(\epsilon_8)\Big(\sup_{0\leq t\leq T}\int_\rea v^2 \,{\rm d}x\Big)\Big(\int_0^T \sup_{\Sigma_2(t)}(\theta-1)^2 \,{\rm d}t \Big)\\
&\leq \epsilon_8 \int_0^T\int_{\Sigma_2(t)}v^2v_x^2 \,{\rm d}x{\rm d}t + C(\epsilon_8)\Big(\sup_{0\leq t\leq T}\int_\rea v^2  \,{\rm d}x\Big) \Big(\int_0^T\sup_{\rea}\big[\theta(t,\cdot)-\frac{3}{2}\big]_{+}^2  \,{\rm d}t\Big)\\
&\leq  \epsilon_8 \int_0^T\int_{\Sigma_2(t)}v^2v_x^2 \, {\rm d}x{\rm d}t + \epsilon_9 \int_0^T\int_{\Sigma_{\frac{3}{2}}(t)} \theta_x^2(t,x) \,{\rm d}x{\rm d}t + C(\epsilon_8, \epsilon_9),
\end{align*}
where in the final line one utilises Eq. \eqref{by parts add}. 
\end{itemize}

Finally, we select $\epsilon_6, \epsilon_7, \epsilon_8, \epsilon_9$ so small that the corresponding terms get absorbed into the left-hand side of Eq. \eqref{eq_intermediate step}.   The proof is completed by putting $K_1, K_2, K_3$ together.

\end{proof}

\section{Completion of the Proof of Theorems   \ref{Theorem for global existence of weak solution} and \ref{thm_large time behaviour}}

With the above preparations, we finally arrive at the stage of proving the main results of the paper, {\it i.e.}, Theorems  \ref{Theorem for global existence of weak solution} and \ref{thm_large time behaviour}, concerning the global existence and large-time behaviour of Eqs. \eqref{eq_mass equation}--\eqref{eq_conditions on id}. 

This final section is organised as follows: First, let us derive some uniform bounds for the higher derivatives of $(u,v,\theta,Z)$. As a by-product, the temperature $\theta$ is uniformly bounded from the above. Then, employing these bounds and investigating the limiting process $T \rightarrow \infty$, we are able to deduce the large-time behaviour, {\it i.e.}, Theorem \ref{thm_large time behaviour}. Thus the uniform lower bound for $\theta$ can be deduced, which  agrees with the physical law that the absolute zero temperature cannot be reached. As both the upper and the lower bounds for $\theta$ are at hand, our local (in time) estimates can be extended globally. Finally, the global existence of weak solutions are derived as a corollary of the  estimates aforementioned.

\begin{lemma}\label{Lemma_bounds on higher derivatives}
There exists a universal constant $C_5$ such that  the following estimate holds for the weak solution on $[0,T)\times\rea$:
\begin{align}\label{aaa}
&\sup_{0 \leq t \leq T} \ir \big( u_x^2 + v_x^2 + \theta_x^2+ Z_x^2 \big) \,{\rm d}x \nonumber\\
 +\:& \int_0^T\ir \big(\theta u_x^2 + u_{xt}^2 + v_{xx}^2 + \theta_{xx}^2 + Z_{xx}^2 +v_t^2 + \theta_t^2 + Z_t^2 \big) \,{\rm d}x{\rm d}t \leq C_5.
\end{align}
Moreover, $\theta$ is uniformly bounded from above:
\begin{equation}
\supt \leq C_5. 
\end{equation}
\end{lemma}

\begin{proof}

Before carrying out the estimates, we notice that the terms in Eq. \eqref{aaa} involving $u_{xt}^2, v_t^2, \theta_t^2, Z_t^2$ are bounded by the other terms in the same equation: This is an immediate consequence of Eqs.  \eqref{eq_mass equation} -- \eqref{eq_reaction equation}. Therefore, we only need to bound the {\em spatial} derivatives, which is shown in the following five steps:
	 
	 \smallskip
	 {\bf 1.} First of all, let us estimate the derivatives of $u$. Substituting the mass equation \eqref{eq_mass equation} in the momentum equation \eqref{eq_momentum equation}, one deduces that
\begin{equation*}
	 v_t +a \big(\frac{\theta}{u}\big)_x = \mu \big(\log(u)\big)_{tx}.
\end{equation*}
Then, multiplying $\big(\log(u)\big)_x$ to both sides, we obtain:
\begin{equation}\label{by parts 4}
\frac{\mu}{2} \Big((\log u)_x^2\Big)_t + a \frac{\theta u_x^2}{u^3} = \Big(v (\log u)_t\Big)_x + \frac{u_x\theta_x}{u^2} + \Big(v\frac{u_x}{u}\Big)_t - \frac{v_x^2}{u}.
\end{equation}
	In view of Theorem \ref{thm_crucial estiamte in theta and u, Li Liang type} and the entropy formula \eqref{eq_entropy ineq}, we integrate over $[0,T]\times\rea$ to get
	\begin{align*}
	&\sup_{0\leq t\leq T} \ir \frac{u_x^2}{u^2} \dx +\int_0^T\ir \theta u_x^2 \,{\rm d}x{\rm d}t\\
	\leq \,&  \epsilon_{10}\int_0^T\ir  \theta u_x^2 \,{\rm d}x{\rm d}t + C(\epsilon_{10}) \int_0^T\ir  \frac{\theta_x^2}{\theta^2}  \,{\rm d}x{\rm d}t + C(\epsilon_{10}) \int_0^T \ir {\theta^2} \,{\rm d}x{\rm d}t\\
	&\: + C(\epsilon_{10}) \int_0^T\ir \frac{v_x^2}{\theta} \,{\rm d}x{\rm d}t + \epsilon_{10} \sup_{0\leq t\leq T} \ir \frac{u_x^2}{u^2}\,{\rm d}x + C(\epsilon_{10})\sup_{0\leq t\leq T} \ir v^2 \,{\rm d}x{\rm d}t\\
	\leq& \epsilon_{10} \sup_{0\leq t\leq T} \ir \frac{u_x^2}{u^2} \,{\rm d}x + \epsilon_{10}\int_0^T\ir  \theta u_x^2 \,{\rm d}x{\rm d}t + C(\epsilon_{10}).
	\end{align*}
So, by choosing suitably small $\epsilon_{10}$, the above estimates give us
\begin{equation}\label{eq_interm estimate for higher derivative bounds, 1}
\sup_{0\leq t\leq T} \ir u_x^2 \,{\rm d}x + \int_0^T\ir \theta u_x^2 \,{\rm d}x{\rm d}t \leq C_6.
\end{equation}

\smallskip
{\bf 2.} Now we estimate the derivatives of $v$ by multiplying $v_{xx}$ to the momentum equation \eqref{eq_momentum equation}. In this way one gets
\begin{equation}\label{by parts 5}
\frac{1}{2} \big(v_x^2 \big)_t + \frac{\mu}{u} (v_{xx})^2 = \big(v_xv_t\big)_x + \mu\frac{v_xu_xv_{xx}}{u^2} - a\frac{v_{xx}\theta_x}{u} + a\frac{\theta u_x v_{xx}}{u^2}.
\end{equation}
Thus we obtain:
\begin{align*}
&\sup_{0\leq t \leq T} \ir v_x^2 \, {\rm d}x + \int_0^T \ir v_{xx}^2\, {\rm d}x{\rm d}t \\
\leq\:& \epsilon_{11} \int_0^T \ir v_{xx}^2 \, {\rm d}x{\rm d}t + C(\epsilon_{11})\int_0^T \ir \frac{v_x^2}{\theta} \,{\rm d}x{\rm d}t  \nonumber\\
+\:& 2C(\epsilon_{11})\Big\{\sup_{[0,T]\times\rea}\theta(\cdot,\cdot)\Big\} \int_0^T \ir \theta u_x^2 \,{\rm d}x{\rm d}t + C(\epsilon_{11})\int_0^T \ir \theta_x^2 \,{\rm d}x{\rm d}t.
\end{align*}
The last three terms on the right-hand side are bounded by the entropy formula \eqref{eq_entropy ineq}, Theorem \ref{thm_crucial estiamte in theta and u, Li Liang type} and Eq.  \eqref{eq_interm estimate for higher derivative bounds, 1} in Step 1 of the same proof. Thus, choosing $\epsilon_{11}$ suitably small, we arrive at the following: \begin{equation}\label{eq_interm estimate for higher derivative bounds, 2}
\sup_{0\leq t \leq T} \ir v_x^2 \,{\rm d}x + \int_0^T \ir v_{xx}^2 \,{\rm d}x{\rm d}t \leq C_7 \; \Big(1+\sup_{[0,T]\times\rea} \theta \Big).
\end{equation}

\smallskip
{\bf 3.} Next, let us estimate the derivatives of $Z$, which is specific to our problem of the reacting mixture. We multiply $Z_{xx}$ to Eq. \eqref{eq_reaction equation} to get
\begin{equation}\label{by parts 6}
\frac{\big(Z_x^2\big)_t}{2} + \frac{d}{u^2} Z_{xx}^2 = \big[(Z_t+ K\phi(\theta) Z) Z_x\big]_x  - 2d\frac{u_x Z Z_x Z_{xx}}{u^3}.
\end{equation}

Now recall that $0 \leq Z \leq 1$ always holds (Lemma \ref{lemma_bdd of Z}); so, thanks to the Sobolev inequality in Eq. \eqref{simple identity}, the following estimates are valid:
\begin{align}
&\sup_{0 \leq t \leq T} \ir Z_x^2(t,x) \,{\rm d}x + \int_0^T\ir Z_{xx}^2 \,{\rm d}x{\rm d}t \nonumber\\
\leq\:&  \epsilon_{12}\int_0^T\ir Z_{xx}^2 \,{\rm d}x{\rm d}t+ C(\epsilon_{12})\Big\{ \sup_{0\leq t \leq T}\ir |u_x|^2 \,{\rm d}x\Big\}\Big\{\sup_{x\in\rea}\int_0^T |Z_x|^2 \,{\rm d}t\Big\}\nonumber\\
\leq\:& \epsilon_{12}\int_0^T\ir Z_{xx}^2 \,{\rm d}x{\rm d}t + C(\epsilon_{12}) \int_0^T \|Z_x(t,\cdot)\|_{L^2(\rea)} \|Z_{xx}(t,\cdot)\|_{L^2(\rea)} \,{\rm d}t \nonumber\\
\leq\:&  \epsilon_{12}\int_0^T\ir Z_{xx}^2 \,{\rm d}x{\rm d}t + \epsilon_{13}\int_0^T\ir Z_{xx}^2(t,x) \,{\rm d}x + C(\epsilon_{12},\epsilon_{13})\int_0^T\ir Z_x^2(t,x) \,{\rm d}x{\rm d}t.
\end{align}

On the other hand, multiplying $Z$ to Eq. \eqref{eq_reaction equation} leads to:
\begin{equation*}
\frac{1}{2} (Z^2)_t + K\phi(\theta)Z^2 + \frac{d}{u^2}(Z_x)^2 = \big(\frac{d}{u^2}ZZ_x\big)_x.
\end{equation*}
As shown in Lemma \ref{lemma_bdd of Z}, the $L^2$ norm of $Z$ decreases in time; thus
\begin{equation}
\int_0^T\ir Z_x^2 \,{\rm d}x{\rm d}t\leq \frac{1}{2}\|Z(T,\cdot)\|^2_{L^2(\rea)}\leq\frac{1}{2} \|Z_0\|_{L^2(\rea)},
\end{equation}
which leads to the conclusion as follows:
\begin{equation}\label{eq_interm estimate for higher derivative bounds, 3}
\sup_{0\leq t \leq T}\ir Z_x^2 \, {\rm d}x + \int_0^T \ir Z_{xx}^2 \,{\rm d}x{\rm d}t + \int_0^T \ir Z_x^2 \,{\rm d}x{\rm d}t \leq C_8.
\end{equation}

\smallskip
{\bf 4.} In this step we  establish the bounds for derivatives of $\theta$. As before, multiplying $\theta_{xx}$ to the temperature equation \eqref{eq_simplified temperature equation} yields:
\begin{align}\label{by parts special}
\frac{1}{2}\big(\theta_x^2\big)_t + \frac{\kappa}{u}\theta_{xx}^2 = \big(\theta_t \theta_x \big)_x 
+ \kappa\frac{\theta_x u_x \theta_{xx}}{u^2} - qK\phi(\theta)Z_x \theta_x +a\frac{\theta}{u}v_x \theta_{xx} - \mu \frac{v_x^2 \theta_{xx}}{u}.
\end{align}

Now, we integrate over $[0,T]\times \rea$ and repetitively use Eq. \eqref{eq_entropy ineq}, Theorem \ref{thm_crucial estiamte in theta and u, Li Liang type}, Eq. \eqref{simple identity}, Young's inequality, 
as well as Eqs. \eqref{eq_interm estimate for higher derivative bounds, 1}\eqref{eq_interm estimate for higher derivative bounds, 2} and \eqref{eq_interm estimate for higher derivative bounds, 3} in the previous steps of the same proof, to derive the following inequality:
\begin{align*}
&\sup_{0\leq t \leq T} \ir \theta_x^2 \,{\rm d}x + \int_0^T\ir \theta_{xx}^2 \,{\rm d}x{\rm d}t\\
\leq&\: C \Big\{ \int_0^T \|\theta_{xx}\|_{L^2(\rea)} \|\theta_x\|_{L^\infty(\rea)} \|u_x\|_{L^2(\rea)}{\rm d}t +  \int_0^T \|Z_x\|_{L^2(\rea)}\|\theta_x\|_{L^2(\rea)} \,{\rm d}t \\
+ \,&  \Big(\supt^{\frac{3}{2}}\Big)\times \int_0^T \|\theta_{xx}\|_{L^2(\rea)} \|\frac{v_x}{\sqrt{\theta}} \|_{L^2(\rea)} \,{\rm d}t +  \int_0^T \|v_x\|_{L^2(\rea)}\|v_x\|_{L^\infty(\rea)}\|\theta_{xx}\|_{L^2(\rea)} \,{\rm d}t\Big\}.
\end{align*}

In the sequel let us bound each of the four terms on the right-hand side of the preceding expression. For the first term, we consider
	\begin{align*}
&\int_0^T \Big(\|\theta_{xx}\|_{L^2(\rea)} \|\theta_x\|_{L^\infty(\rea)} \|u_x\|_{L^2(\rea)}\Big) \,{\rm d}t \\
\leq\:& \sup_{0\leq t \leq T} \|u_x(t,\cdot)\|_{L^2(\rea)}\int_0^T {\Big(\int_{\rea} \theta_{xx}^2dx\Big)^{\frac{3}{4}} }\Big(\int_{\rea}\theta_x^2 \dx\Big)^{\frac{1}{4}} \dt\\
\leq\:& \epsilon_{14}\int_0^T\int_{\rea}\theta_{xx}^2 \,{\rm d}x{\rm d}t + C(\epsilon_{14})\int_0^T\int_{\rea} \theta_x^2 \,{\rm d}x{\rm d}t\\
\leq\:&\epsilon_{14}\int_0^T\int_{\rea}\theta_{xx}^2 \,{\rm d}x{\rm d}t + C(\epsilon_{14}),
	\end{align*}
where we have used Young's inequality
\begin{equation*}
ab \leq \frac{3a^{4/3}}{4} + \frac{b^4}{4} \qquad \text{ for } a,b>0,
\end{equation*}
as well as Eq. \eqref{eq_interm estimate for higher derivative bounds, 1} and the entropy formula \eqref{eq_entropy ineq}.

The second term is easily bounded as follows:
	\begin{align*}
	\int_0^T\|Z_x\|_{L^2(\rea)}\|\theta_x\|_{L^2(\rea)} \,{\rm d}t \leq \epsilon_{15} \sup_{0\leq t \leq T} \ir\theta_x^2 \dx + C(\epsilon_{15}).
	\end{align*}

For the third term we compute as follows:
	\begin{align*}
	& \Big(\supt^{\frac{3}{2}}\Big) \int_0^T \|\theta_{xx}\|_{L^2(\rea)} \|\frac{v_x}{\sqrt{\theta}} \|_{L^2(\rea)} \,{\rm d}t \\
	 \leq\:& \frac{1}{2}\supt^{3}+\frac{1}{2}\Big\{\int_0^T \|\theta_{xx}\|_{L^2(\rea)} \|\frac{v_x}{\sqrt{\theta}} \|_{L^2(\rea)} \,{\rm d}t\Big\}^2 \\
	 \leq\:& \frac{1}{2}\supt^{3}+\frac{1}{2}\Big(\epsilon_{16}\int_{0}^T\ir \theta_{xx}^2 \,{\rm d}x{\rm d}t + C(\epsilon_{16})\int_0^T\ir \frac{v_x^2}{\theta} \,{\rm d}x{\rm d}t \Big)\\
	 \leq\:& \epsilon_{16}\int_{0}^T\ir \theta_{xx}^2 \, {\rm d}x{\rm d}t + C(\epsilon_{16}) \supt^3.
	\end{align*}
	
Finally, for the fourth term, we employ again Eq. \eqref{simple identity} to derive that
	\begin{align*}
&\int_0^T \Big\{\|v_x\|_{L^2(\rea)}\|v_x\|_{L^\infty(\rea)}\|\theta_{xx}\|_{L^2(\rea)}\Big\}\,{\rm d}t \\
\leq\:& C\int_0^T \Big\{\|v_x\|_{L^2(\rea)}\|v_x\|_{H^1(\rea)}\|\theta_{xx}\|_{L^2(\rea)}\Big\}\,{\rm d}t\\
\leq\:& C\Big(\sup_{0\leq t \leq T}\|v_x\|_{L^2(\rea)}\Big)\Big(\int_0^T\|v_x\|_{H^1(\rea)}\|\theta_{xx}\|_{L^2(\rea)}{\rm d}t\Big)\\
\leq\:& C\sup_{0\leq t\leq T}\int_\rea v_x^2(t,x) \dx +  \Big(\int_0^T\|v_x\|_{H^1(\rea)}\|\theta_{xx}\|_{L^2(\rea)} \dx {\rm d}t\Big)^2\\
\leq\:& C\sup_{0\leq t\leq T}\int_\rea v_x^2(t,x) \dx + C(\epsilon_{17}) \int_0^T \int_\rea \big(v_x^2+v_{xx}^2\big)\,{\rm d}x{\rm d}t + \epsilon_{17}\int_0^T\int_\rea \theta_{xx}^2 \,{\rm d}x{\rm d}t\\
\leq\:&  \epsilon_{17}\int_0^T\int_\rea \theta_{xx}^2 \,{\rm d}x{\rm d}t + C(\epsilon_{17})\big\{1+\supt\big\},
	\end{align*}
which again is based on the Cauchy-Schwartz inequality and the entropy formula \eqref{eq_entropy ineq}.

Therefore, using the previous estimates, we choose suitable $\epsilon_{14}, \epsilon_{15}, \epsilon_{16}$ and $\epsilon_{17}$ to get:
\begin{equation}\label{eq_interm estimate for higher derivative bounds, 4}
\sup_{0\leq t \leq T} \ir \theta_x^2 {\rm d}x + \int_0^T \ir \theta_{xx}^2 {\rm d}x{\rm d}t \leq C_9 \Big(1+ \supt+\supt^3\Big).
\end{equation}

\smallskip
{\bf 5.} Finally we conclude the uniform upper boundedness of $\theta$ in space-time. Notice that, by the Sobolev inequality \eqref{simple identity}, for any $0\leq t \leq T$ there holds
\begin{equation}\label{eq_interm estimate for higher derivative bounds, 5}
\|(\theta-2)_{+}(t,\cdot)\|^2_{C(\rea)} \leq \|(\theta-2)_{+}(t,\cdot)\|_{L^2(\rea)} \|\theta_x(t,\cdot)\|_{L^2(\rea)}.
\end{equation}
Then, using Theorem \ref{thm_crucial estiamte in theta and u, Li Liang type}, it can be deduced that $\|(\theta-2)_{+}(t,\cdot)\|_{L^2(\rea)}\leq \sqrt{C_1}$ for any $t$, while $\|\theta_x(t,\cdot)\|_{L^2(\rea)}$ is estimated by Eq. \eqref{eq_interm estimate for higher derivative bounds, 4}. Hence we get 
\begin{equation*}
\|(\theta-2)_{+} (t,\cdot)\|^2_{C(\rea)} \leq C_{10} \Big(1+ \supt^{\frac{1}{2}}+\supt^\frac{3}{2}\Big).
\end{equation*}
In particular, by comparing the growth rate at infinity, we get:
\begin{equation}\label{eq_uniform upper bound for theta}
\sup_{[0,T]\times \rea}\theta(\cdot,\cdot) \leq C_{11}.
\end{equation}
Thus, putting together the estimates in Eqs. \eqref{eq_interm estimate for higher derivative bounds, 1}\eqref{eq_interm estimate for higher derivative bounds, 2}\eqref{eq_interm estimate for higher derivative bounds, 3}\eqref{eq_interm estimate for higher derivative bounds, 4}\eqref{eq_uniform upper bound for theta}, the proof is complete.

\end{proof}

Based on Lemma \ref{Lemma_bounds on higher derivatives}, we are now ready to establish the global existence and the large-time behaviour of the weak solutions, which are the main results of the paper:

\begin{proof}[Proof of Theorems \ref{Theorem for global existence of weak solution} and \ref{thm_large time behaviour}]

The arguments are divided in three steps. 

\smallskip
{\bf 1.} First, let us prove the large-time behaviour under the temporary assumption $(\clubsuit)$ introduced at the beginning of \S 3, namely that the reaction rate function $\phi \in C^1(\rea) \cap L^\infty(\rea)$. In this case, due to the uniform estimate established in Lemma \ref{Lemma_bounds on higher derivatives}, sending $T \rightarrow \infty$ gives us:
\begin{equation}
\int_0^\infty \Bigg\{\Big|\frac{d}{dt}\|v_x(t,\cdot)\|_{L^2(\rea)} \Big| + \Big|\frac{d}{dt}\|\theta_x(t,\cdot)\|_{L^2(\rea)} \Big| + \Big|\frac{d}{dt}\|Z_x(t,\cdot)\|_{L^2(\rea)} \Big|\Bigg\}\, {\rm d}t \leq C_{12}.
\end{equation}
Hence, we have 
\begin{equation}\label{derivatives tend to zero}
\big\|(v_x, \theta_x, Z_x)(t,\cdot)\big\|_{L^2(\rea)} \longrightarrow 0 \qquad\text{ as } t \rightarrow \infty.
\end{equation}
From here we immediately deduce that
\begin{align}\label{asymptotics for u}
v^2(t,x) &\leq \|v_x(t,\cdot)\|_{L^2(\rea)}\|v(t,\cdot)\|_{L^2(\rea)} \nonumber\\
&\leq \sqrt{2qE_0} \|v_x(t,\cdot)\|_{L^2(\rea)} \longrightarrow 0,
\end{align}
which is valid for any $x\in\rea$, in view of the Sobolev inequality \eqref{simple identity}. 

Next, the asymptotic for $\theta$ is obtained similarly: Thanks to that $\supt \leq 3+C_{10}$ ({\it cf.} Step 5 in the proof of Lemma \ref{Lemma_bounds on higher derivatives}) and Eq. \eqref{estimate on theta-1 square}, we have
\begin{equation}\label{asymptotics for theta}
(\theta(t,x)-1)^2 \leq \|(\theta(t,\cdot)-1)\chi_{\rea \setminus \Sigma_{2}(t)}\|_{L^2(\rea)}\|\theta_x(t,\cdot)\|_{L^2(\rea)} \longrightarrow 0,
\end{equation}
where Dominated Convergence Theorem is used. Also, Lemma \ref{Lemma_bounds on higher derivatives} leads to the following: \begin{equation*}
\sup_{0\leq t \leq T}\ir u_x^2 \,{\rm d}x + \int_0^T\ir \theta u_x^2 \,{\rm d}x{\rm d}t \leq C_5.
\end{equation*}
As we have already established the uniform boundedness of $\theta$, it follows that
\begin{equation*}
\int_0^\infty \Big|\frac{d}{dt}\|u_x(t,\cdot)\|_{L^2(\rea)}\Big| \,{\rm d}t \leq C_{12}.
\end{equation*}

Moreover, the subsequent  result holds:
\begin{equation}\label{asymptotics for v}
\big\{u(t,x) - 1\big\}\longrightarrow 0\qquad \text{ uniformly in space-time}.
\end{equation}
Indeed, by the entropy inequality \eqref{eq_entropy ineq} and the uniform bound on $u$ ({\it cf.} Theorem \ref{thm_uniform upper and lower bds for v}), 
\begin{equation}
\sup_{0\leq t \leq T} \ir (u-1)^2 \,{\rm d}x \leq C\sup_{0\leq t \leq T} \ir \big(u-1-\log(u)\big)\, {\rm d}x \leq C_{13}.
\end{equation}
Thus, via precisely the same arguments for $v$ and $\theta$ as above \eqref{asymptotics for v} is proved.

Finally, to control the combustion term $Z$ (which is specific to our problem), we integrate by parts to derive that
\begin{align}\label{asymptotics for Z}
Z^{\frac{3}{2}} (t,x) &= -\int_x^\infty\big[Z^{\frac{3}{2}}\big]_x(t,\xi) \,{\rm d}\xi \nonumber\\
&\leq \frac{3}{2} \ir Z^{\frac{1}{2}}(t,x) \big|Z_x(t,x)\big|\,{\rm d}x\nonumber\\
&\leq \frac{3}{2}\|Z_x\|_{L^2(\rea)} \Big(\ir Z(t,x)\,{\rm d}x \Big)^{\frac{1}{2}} \longrightarrow 0,
\end{align}
thanks to Eqs. \eqref{derivatives tend to zero} and \eqref{eq_Z inequality}. Therefore, collecting the estimates in Eqs. \eqref{asymptotics for u}\eqref{asymptotics for theta}\eqref{asymptotics for v}\eqref{asymptotics for Z}, we see that the proof for Theorem \ref{thm_large time behaviour} is now complete, provided $\phi\in C^1(\rea)\cap L^\infty(\rea)$.

\smallskip
{\bf 2.}
In this step we establish the uniform lower bound for $\theta$, based on the large-time behaviour established in  Step 1 of the same proof for $C^1$ reaction rate functions.

For this purpose, we first obtain a lower bound for $\theta$ up to some given time $T_\ast >0$ on the compact domain $[-L,L]$ for some finite number $L>0$. Let us denote by $\zeta:=\theta^{-1}$. Then, multiplying $(-\theta^{-2})$ to the temperature equation \eqref{eq_temperature equation}, we arrive at the following evolution equation for $\zeta$:

\begin{equation}\label{eq_evolution for zeta=inverse theta}
\zeta_t +2\kappa \frac{\zeta_x^2}{u\zeta} +  \mu \frac{\zeta^2v_x^2}{u} =  \kappa\frac{\zeta_{xx}}{u}-\kappa \frac{\zeta_x u_x}{u^2} + a\frac{\zeta}{u}v_x - \zeta^2qK\phi(\theta)Z.
\end{equation}
Completing the squares and writing the first two terms on right-hand side in the full divergence form, we obtain that
\begin{equation*}
\zeta_t + 2\kappa\frac{\zeta_x^2}{u\zeta} + \frac{\mu}{u} \Big[ {\zeta v_x} - \frac{a}{2\mu }\Big]^2 + \zeta^2 q K\phi(\theta)Z= \kappa \Big(\frac{\zeta_x}{u}\Big)_x + \frac{a^2}{4\mu u}.
\end{equation*}

Then, we restrict to the finite spatial interval $[-L,L]$, and multiply $(2p\zeta^{2p-1})$ to the previous equation with $p>\frac{3}{2}$. Integrating by parts and using the periodic boundary condition on $[-L,L]$, we arrive at
\begin{equation*}
\frac{d}{dt}\int_{-L}^L \zeta^{2p}(t,x) \dx \leq \frac{a^2p}{2\mu}\int_{-L}^L \frac{\zeta(t,x)^{2p-1}}{u(t,x)} \,{\rm d}x.
\end{equation*}
Now, applying H\"older's inequality together with the uniform lower bound $u \geq C_0^{-1}$ in Theorem \ref{thm_uniform upper and lower bds for v}, one deduces:
\begin{equation}
2p\|\zeta\|_{L^{2p}([-L,L])}^{2p-1} \times \frac{d}{dt}\|\zeta\|_{L^{2p}([-L,L])} \leq 2p \frac{a^2}{4\mu C_0} \big\|\zeta^{2p-1}\big\|_{L^{\frac{2p}{2p-1}}([-L,L])} (2L)^{\frac{1}{2p}}.
\end{equation}
Here, it is crucial to choose $L$ {\em depending on $p$}: Indeed, we take 
\begin{equation}
L=2^{2p-1},
\end{equation}
then $L\rightarrow\infty$ as $p\rightarrow\infty$, while $(2L)^{\frac{1}{2p}} = 2$. The previous estimate thus becomes:
\begin{equation*}
\frac{d}{dt}\|\zeta\|_{L^{2p}([-L,L])} \leq \frac{a^2}{2\mu C_0}, 
\end{equation*}
which is a uniform estimate in $L$ and $p$. Thus, for any fixed $T_{\ast}>0$, we can send $p, L$ to infinity and apply the Gr\"{o}nwall lemma to conclude that 
\begin{equation}
\zeta(t,x) \leq C e^{CT_\ast}.
\end{equation}
Equivalently, we have just established:
\begin{equation}\label{eq_local lower bound for theta}
\inf_{[0,T_\ast]\times\rea} \theta \geq C^{-1}e^{-CT_\ast},
\end{equation}
which is a space-time uniform lower bound for $\theta$ up to time $T_{\ast}$.

Finally, to promote the local (in time) bound to a global bound, we make use of the result in Step 1 above: there we have shown that  $\theta \rightarrow 1$ uniformly as $t \rightarrow\infty$. As a result, choose a $T_\ast \in (0,\infty)$ such that $0.99\leq\theta(t,x)\leq 1.01$ whenever $t \geq T_\ast$ and $x\in\rea$. Thus, together with the {\em local} lower bound of $\theta$ (Eq. \eqref{eq_local lower bound for theta}), we readily conclude the {\em global} lower bound for $\theta$. Now  we are able to conclude the proof of Theorems \ref{Theorem for global existence of weak solution} and \ref{thm_large time behaviour}, subject to the condition $(\clubsuit)$.

\smallskip
{\bf 3.} Finally let us remove the condition $(\clubsuit)$ and establish the theorems for generic discontinuous functions obeying the Arrenhius' law.

For this purpose, we take a discontinuous $\phi$ obeying the Arrhenius Law and mollify it with
\begin{equation}
\phi_\eta(\theta) : = (J_\eta \ast \phi)(\theta),
\end{equation}
where  $J$ is the standard mollifier, namely $J \in C^\infty(\rea), \ir Jdx =1, J_\eta(\theta)= \frac{1}{\eta}J(\frac{\theta}{\eta})$. It always holds that $\phi_\eta(\theta) \rightarrow \phi(\theta)$ in $L^p(\rea)$, for any $p \in [1,\infty)$, as $\eta \rightarrow 0^{+}$. It is crucial to notice that
\begin{equation}
\|\phi_\eta(\theta)\|_{C(\rea)} \leq \|\phi(\theta)\|_{C(\rea)} \leq M \qquad \text{ for any } \eta >0,
\end{equation}
even if the $C^1$-norm of $\phi_\eta(\theta)$ may blow up as $\eta\rightarrow 0^{+}$.
	
	Now, by a careful examination of \S\S 3--5, all the estimates derived therein are independent of $\delta$. Thus, for each $\eta >0$, we argue as in \S\S 3-5 with respect to Eqs. \eqref{eq_mass equation}\eqref{eq_momentum equation}\eqref{eq_temperature equation}\eqref{eq_reaction equation}, with $\phi(\theta)$ replaced by $\phi_\eta(\theta)$. In this manner, we obtain a global weak solution $(u_\eta, v_\eta, \theta_\eta, Z_\eta)$ for each $\eta$, which verifies the uniform estimates independently of $\eta$. Hence, by a standard compactness argument in the Sobolev space $[H^1(\rea)]^4$,  a subsequence (still labelled as $\{(v_\eta, u_\eta, \theta_\eta, Z_\eta)\}_{\eta>0}$) which converges to a global weak solution to Eqs. \eqref{eq_mass equation} -- \eqref{eq_conditions on id}, with the discontinuous reaction rate function $\phi(\theta)$ satisfying the Arrehnius' law. 
	
	Therefore, the proof of Theorems \ref{Theorem for global existence of weak solution} and \ref{thm_large time behaviour} is now complete.

\end{proof}

At the end of the paper we make four concluding remarks:

\begin{enumerate}
\item
	First of all, the physical meaning of the results in this paper is natural: For a one-dimensional reacting mixture on unbounded domains, if the far-field condition  is imposed as in Eq. \eqref{eq_far-field condition}, then the chemical reaction will occur and proceed toward completion as time approaches infinity, regardless of the detailed structure of the reaction-rate function. In this process, the density and temperature of the reacting mixture will be uniformly bounded away from zero and infinity.
\item
	Next, in this work we only consider the global existence of weak solutions,  but we have not addressed the issues of {\em classical} ({\it e.g.} $C^\infty$ or $C^{2,\alpha}$) solutions at all. Indeed, due to the discontinuity of $\phi(\theta)$ (viewed as a part of the coefficients of the evolution equation \eqref{eq_temperature equation}), we should not expect the existence of classical solutions.
\item
The results in the paper can be extended to several other types of boundary conditions. 
 For example, let us consider the domain to be the half line $\Omega = [0, \infty)$, with the same far-field condition at $\infty$:  
\begin{equation*}
\lim_{x\rightarrow \infty} (u,v,\theta, Z) = (1,0,1,0) \qquad \text{ for all } t \geq 0.
\end{equation*}
At $x=0$ we can impose the 
{\bf impermeability + thermal insulation condition (I)}:
\begin{equation}
v(t,0)=0; \qquad \theta_x(t,0)=0; \qquad Z(t,x)=0 \text{ or } Z_x(t,0)=0,
\end{equation}
or the {\bf impermeability + constant source condition (II)}: 
\begin{equation}
v(t,0)=0;\qquad \theta(t,0)=1;\qquad Z(t,x)=0 \text{ or } Z_x(t,0)=0.
\end{equation}
These boundary conditions are also considered in \cite{LL14} for one-dimensional heat-conducting compressible fluids without reaction terms.

\noindent Here, we claim that the same statements for Theorems \ref{Theorem for global existence of weak solution} and \ref{thm_large time behaviour} remain valid, subject to boundary conditions{ \bf (I)} or {\bf (II)}. This can be proved by precisely the same arguments, as long as the integration by parts arguments still hold. Indeed, Eqs. \eqref{by parts 1} \eqref{by parts 2} \eqref{by parts 3} \eqref{by parts add} \eqref{by parts 4} \eqref{by parts 5} and \eqref{by parts 6} remain valid; also, under the condition {\bf (I)}, Eq. \eqref{by parts special} stays the same, while subject to the condition {\bf (II)} we can make simple modifications to  recover Eq. \eqref{eq_interm estimate for higher derivative bounds, 4}. 
 
\item
In the end, we emphasize that the arguments in \cite{KS77}  for the lower boundedness of temperature are not valid on unbounded domains ($u^{-1} \notin L^p(\rea)$ for $p \geq 1$), and the arguments in \cite{LL14} for the large-time behaviour cannot be applied without modifications in presence of the $Z$ term. In our work, new estimates have been  developed to cope with unbounded domains and the chemical reaction terms.
\end{enumerate}





\bigskip
\noindent
{\bf Acknowledgement}
Siran Li's research was supported in part by the UK EPSRC Science and Innovation award to the Oxford Centre for Nonlinear PDE (EP/E035027/1).



\bigskip

\begin{thebibliography}{99}

%

\bibitem{LL14}
J. Li and Z. Liang, {Some uniform estimates and large-time behavior of solutions to one-dimensional compressible Navier-Stokes system in unbounded domains with large data},
\textit{Arch. Rat. Mech. Anal.} \textbf{220} (2016), 1195--1208.

\bibitem{C92}
G.-Q. Chen, {Global solutions to the compressible Navier-Stokes equations for a reacting mixture}, \textit{SIAM Jour. Math. Anal.}, \textbf{23} (1992), 609--634. 

\bibitem{CHT03}
G.-Q. Chen, D. Hoff, and K. Trivisa, {Global solutions to a model for exothermically reacting, compressible flows with large discontinuous initial data}, \textit{Arch. Rat. Mech. Anal.}, \textbf{166} (2003), 321--358.


\bibitem{CK}
G.-Q. Chen and M. Krakta, {Global solutions to the Navier-Stokes equations for compressible heat-conducting flow with symmetry and free boundary}, \textit{Comm. PDE}, \textbf{27} (2002), 907--943.


\bibitem{CW}
G.-Q. Chen and D.~H. Wagner, {Global entropy solutions to exothermically reacting, compressible Euler equations}, \textit{J. Diff. Eqn.}, \textbf{191} (2003), 277--322.

\bibitem{DT}
D. Donatelli and K. Trivisa, {On the motion of a viscous compressible readiactive-reacting gas}, \textit{Comm. Math. Phys}, \textbf{265} (2006), 463--491.


\bibitem{DZ}
B. Documet and A. Zlotnik, {Lyapunov functional method for 1D radiactive and reactive viscous gas dynamics}, \textit{Arch. Rat. Mech. Anal}, \textbf{177} (2005), 185--229.

\bibitem{HLW}
X. Huang, J. Li and Y. Wang, {Serrin-type blowup criterion for full compressible Navier-Stokes system}, \textit{Arch. Rat. Mech. Anal.}, \textbf{207} (2013), 303-316.

\bibitem{J99}
S. Jiang, {Large-time behavior of solutions to the equations of a one-dimensional viscous polytropic ideal gas in unbounded domains}, \textit{Comm. Math. Phys}, \textbf{200} (1999), 181--193.

\bibitem{J02}
S. Jiang, {Remarks on the asymptotic behaviour of solutions to the compressible Navier-Stokes equations in the half-line}, \textit{Proc. R. Soc. Edinb. Sect. A}, \textbf{132} (2002), 627--638.

\bibitem{K82}
A.~V. Kazhikhov, {Cauchy problem for viscous gas equations}, \textit{Siber. Math. Jour.}, \textbf{23} (1982), 44--49.

\bibitem{KS77}
A.~V. Kazhikhov and V.~V. Shelukhin, {Unique global solution with respect to time of initial-boundary value problems for one-dimensional equations of a viscous gas: (PMM vol. 41, n=2, 1977, pp. 282--291)}, \textit{Jour. Appl. Math. Mech.}, \textbf{41} (1977), 273--282.

\bibitem{SK81}
V.~A. Solonnikov and A.~V. Kazhikhov, {Existence theorems for the equations of motion of a compressible viscous fluid}, 
\textit{Ann. Rev. Fluid Mech.} \textbf{13} (1981), 79--95.


\bibitem{wang}
D. Wang, {Global solutions for the mixture of real compressible reacting flows in combustion}, \textit{Comm. Pure Appl. Anal.}, \textbf{3} (2004), 775--790.

\end{thebibliography}
\end{document}